\documentclass[a4paper]{article}

\usepackage[english]{babel}

\usepackage[utf8]{inputenc}
\usepackage[utf8]{inputenc}
\usepackage{amsmath}
\usepackage{graphicx}
\usepackage{amssymb}
\usepackage{amsthm}
\usepackage{tikz-cd}
\usepackage{mathrsfs}
\usepackage[colorinlistoftodos]{todonotes}
\usepackage{enumitem}
\usepackage{yfonts}
\usepackage{ dsfont }
\usepackage{MnSymbol}
\usepackage{slashed}

\newtheorem{thm}{Theorem}[section]
\newtheorem{lem}[thm]{Lemma}

\theoremstyle{definition}

\newtheorem{cor}[thm]{Corollary}

\newtheorem{rmk}[thm]{Remark}
\newtheorem{prop}[thm]{Proposition}

\newtheorem*{Question}{Question}
\newtheorem*{Notation}{Notation}

\newtheorem*{Acknowledgement}{Acknowledgement}

\newcommand{\ie}{\emph{i.e.} }
\newcommand{\cf}{\emph{cf.} }

\newcommand{\R}{\mathbb{R}}
\newcommand{\C}{\mathbb{C}}

\newcommand{\norm}[1]{\left\lVert#1\right\rVert}
\newcommand{\Lap}{\Delta}

\newcommand{\HGF}{\,_2F_1}

\def\XXint#1#2#3{{\setbox0=\hbox{$#1{#2#3}{\int}$ }
		\vcenter{\hbox{$#2#3$ }}\kern-.6\wd0}}

\title{Complete Calabi-Yau metrics in the complement of two divisors}

\author{Tristan C. Collins and Yang Li}
\date{}

\begin{document}

\maketitle

\begin{abstract}
We construct new complete Calabi-Yau metrics on the complement of an anticanonical divisors $D$ in a Fano manifold of dimension at least three, when $D$ consists of two transversely intersecting smooth divisors. The asymptotic geometry is modeled on a generalization of the Calabi ansatz, related to the non-archimedean Monge-Amp\`ere equation. 
\end{abstract}

\section{Introduction}

Following Yau's solution of the Calabi conjecture \cite{Yau}, the analytic approach to the construction of complete Calabi-Yau metrics on noncompact manifolds was initiated by the seminal works of Tian-Yau \cite{TianYau, TianYau2} who proved the following fundamental theorem

\begin{thm}[Tian-Yau, \cite{TianYau}]
Given a smooth irreducible anticanonical divisor $D$ on a smooth Fano manifold $\bar{X}$, then the complement $X=\bar{X}\setminus D$ admits a complete Calabi-Yau metric.
\end{thm}

The asymptotic geometry, from the holomorphic viewpoint, is modelled on the normal bundle $\mathcal{O}(D)|_D= -K_X|_D\to D$. This total space carries a model metric known as the \textbf{Calabi ansatz}, which specifies the asymptotic behaviour of the desired Calabi-Yau metric. The main gist of Tian and Yau's work, which is improved later by Hein \cite{Heinthesis} among others, is that once we have a good asymptotic ansatz, then running a non-compact version of Yau's proof of the Calabi conjecture would produce an actual Calabi-Yau metric on $X$.

In the current paper we are motivated by the following well-known question, which dates back to the work of Tian-Yau;
\begin{Question}
	(`\textbf{Tian-Yau problem}')
 Let $\bar{X}$ be a smooth $n$-dimensional Fano manifold, and $D=\sum D_i$ be an anticanonical divisor, which we assume to be simple normal crossing, with only multiplicity one components. Let $X=\bar{X}\setminus D$, then $X$ has a nowhere vanishing holomorphic volume form $\Omega$. When does $X$ admit a complete Calabi-Yau metric?
\end{Question}

This setup has plenty of examples, for instance $\bar{X}$ can be $\mathbb{CP}^{n}$, and $D$ is a multiplicity one simple normal crossing divisor in $\mathcal{O}(n+1).$ Morally speaking, we are asking `what would happen to the Tian-Yau construction once $D$ breaks up into several components'. Our main result is a solution of the Tian-Yau problem in the case of two proportional divisors.

\begin{thm}
 Let $\bar{X}$ be a smooth $n$-dimensional Fano manifold with $n\geq 3$, whose anticanonical bundle is $(d_1+d_2)L_0$ for some positive line bundle $L_0$, and $d_1, d_2$ are positive integers. Let $D_1, D_2$ be two transversally intersecting smooth divisors in the linear system associated to $d_1L_0$ and $d_2L_0$ respectively. Then $X=\bar{X}\setminus D_1\cup D_2$ admits a complete Calabi-Yau metric. 
\end{thm}

More precisely, the holomorphic geometry in the \textbf{generic region} near infinity is modeled on the total space of the bundle $d_1L_0\oplus d_2L_0\to D_1\cap D_2$. We will explain that there is a \textbf{generalized Calabi ansatz} on (an open subset of) this total space, which is associated with a PDE we call the \textbf{non-archimedean Monge-Amp\`ere equation}. This story has its origin in the context of polarized degenerations, non-archimedean geometry and the SYZ conjecture \cite{NASYZ}.
In our specialized setting, this equation can be dimensionally reduced to an ODE, whose solutions can be rather explicitly given in terms of hypergeometric functions.

The infinity of $X$ also contains \textbf{non-generic regions}, which complex geometrically correspond to the neighbourhood of $D_1\setminus D_2$ (resp. $D_2\setminus D_1$) inside $\bar{X}$. Notice that $D_1\cap D_2$ is an anticanonical divisor inside the Fano manifold $D_1$ (resp. $D_2$), and the metric geometry in the non-generic region involves a fibration by the Tian-Yau metrics on $D_1\setminus D_2$ (resp. $D_2\setminus D_1$) with some power law scalings.
Matching the generalized Calabi ansatz with the metric behaviour in the non-generic region involves a nontrivial ODE matching problem.

The main strategy to construct the Calabi-Yau metric, is to first produce an approximate metric ansatz, and after suitably improving the decay rate of the volume form error, we appeal to Tian-Yau-Hein's existence package. This strategy has been used in a number of recent works, notably \cite{Heingravitational}\cite{LiC3}\cite{LiTaubNUT}\cite{Ronan}\cite{Gabor}.

Our construction suggests that the Tian-Yau problem has an \textbf{inductive structure} in terms of the depth of intersections of the divisors $D_i$ on $\bar{X}$. For instance, in the case of three divisors $D_1, D_2, D_3$ with nonempty intersections, we expect the Tian-Yau problem on $D_1\setminus (D_2\cup D_3)\cap D_1$ (and the cyclic permutations) to appear in the non-generic region at infinity. The generic region would involve a generalized Calabi ansatz metric, associated with a more complicated non-archimedean Monge-Amp\`ere equation, whose boundary conditions are prescribed by the need to match with the nongeneric regions. We think the principal remaining difficulty is then to solve the non-archimedean Monge-Amp\`ere equation, which can in general no longer be reduced to an ODE as we increase the number of divisors. Another interesting direction is to further develop the link with non-archimedean geometry (\cf section \ref{nonarchimedeanmeaning}).

The organization of this paper is as follows. Section \ref{GeneralizedCalabiansatzODE} contains most of the geometric aspects. It introduces the generalized Calabi ansatz, the ODE reduction of the non-archimedean Monge-Amp\`ere equation, and the boundary conditions. We also discuss further relations to the literature and future directions. Section \ref{MoreonODE} explicitly solves the ODE and implements the matching problem. Section \ref{initialerror} is concerned with producing an approximate metric ansatz, and the rather technical issue of error estimates. Section \ref{WeightedSobolevCY} reviews the Tian-Yau-Hein existence package,  and hammers a few final nails.

\begin{Acknowledgement}
Y.L is a current Clay Research Fellow and a CLE Moore Instructor at MIT. He thanks S. Sun for related discussions in the past.
T.C.C is supported in part by NSF CAREER grant DMS-1944952 and an Alfred P. Sloan Fellowship. 
\end{Acknowledgement}

\section{Generalized Calabi ansatz and ODE reduction}\label{GeneralizedCalabiansatzODE}

\subsection{The generalized Calabi ansatz}\label{generalizedCalabiansatz}

We shall describe an approximate ansatz for producing Calabi-Yau metrics, which simultaneously generalizes the Calabi ansatz on the total space of a positive line bundle over a compact Calabi-Yau manifold, and the semiflat metrics coming from torus invariant dimensional reductions of Calabi-Yau metrics. A very closely related ansatz in the context of polarized algebraic degenerations, was discovered in \cite{NASYZ} in an attempt to give a conjectural differential geometric interpretation to the non-archimedean Monge-Amp\`ere equation. Some of our terminologies will therefore reflect the non-archimedean origin of this ansatz.

Let $L_1, \ldots L_m$ be positive line bundles over an $(n-m)$-dimensional compact Calabi-Yau manifold $Y$ with positively curved smooth Hermitian metrics $h_{L_1}, \ldots h_{L_m}$, and let $Z$ be the total space of $L_1\oplus \ldots \oplus L_m\to Y$. Let $r_i$ be the radius distance function on $L_i$. In local holomorphic trivializations of $L_i$, the Hermitian metric on $L_i$ can be written in terms of local potentials as $h_{L_i}= e^{-2\phi_i}$, and the radius distance $r_i=|\xi_i| e^{-\phi_i}$, where $\xi_i$ are the local fibre coordinates on $L_i$. Our task is to construct approximate Calabi-Yau metrics on the region $\{ 0<r_i\ll 1, \forall i \}$ inside $Z$, on the line bundle $L\to Z$ obtained by pulling back a (semi-positive) line bundle $L\to X$. For our main intended applications, $L$ is in fact trivial, and $m=2$. An important conceptual point is that these ansatz metrics will be incomplete in the $r_i\sim 1$ region. Their aim is to provide local metric models on the generic regions of complete Calabi-Yau manifolds, and the global problem would also involve the nontrivial step of finding partial completions of these ansatz metrics.
Concretely, this incompleteness means in the $r_i\sim 1$ region this particular ansatz breaks down, and one must find some alternative ansatz.

It would be helpful to keep in mind that the metric will look like an iterated fibration. On the smallest scale, we have the $T^m$ tori, coming from the circle directions of the line bundles $L_1,\ldots L_m$. These are fibred over compact manifolds diffeomorphic to $Y$, which are close to being Calabi-Yau with length scale much bigger than the tori, and this torus fibration structure is in turn fibred over $m$ noncompact real directions, corresponding roughly to the $\log r_i$ variables. The length scale of the base is much larger than the intermediate length scale of the $Y$-fibres.

\begin{Notation}
	Our convention is $d=\partial+ \bar{\partial}$,  $d^c= \frac{ \sqrt{-1} }{2\pi} (-\partial+ \bar{\partial})$, so $dd^c= \frac{\sqrt{-1}}{\pi}\partial \bar{\partial}$. Given a Hermitian metric $h$ on a line bundle $L$, its curvature form is $-dd^c\log h^{1/2}$ in the class $c_1(L)$.
\end{Notation}

Let $h_L$ be a smooth Hermitian metric on $L\to Y$, which we pull back to $L\to Z$. For our main applications, $L$ is trivial and $h_L=1$. To zeroth approximation, we try the ansatz (which shall be improved later)
\begin{equation*}
h_1= h_L \exp\left( -2 u( -\log r_1, \ldots ,-\log r_m)   \right)
\end{equation*}
where $u$ is some smooth convex function, which should be thought of as the leading order K\"ahler potential in a particular normalisation convention. Write $x_i=-\log r_i$, where the sign is chosen so that $x_i>0$ in the region of interest.  We calculate the K\"ahler metric (the positivity is not automatic, and amounts to an extra assumption)
\[
\begin{split}
& -dd^c \log h_1^{1/2}= -dd^c \log h_L^{1/2} + dd^c u
\\
=& -dd^c \log h_L^{1/2} + %\frac{ \sqrt{-1}}{4\pi} 
\sum \frac{\partial^2 u}{\partial x_i \partial x_j} d\log r_i \wedge d^c \log r_j-  \sum \frac{\partial u}{\partial x_i} dd^c \log r_i.
\end{split}
\]
Recall in local coordinates $r_i= |\xi_i| e^{-\phi_i}$. The Hessian term contains a term
\[
\frac{ \sqrt{-1}}{4\pi} \sum \frac{\partial^2 u}{\partial x_i \partial x_j} d\log \xi_i \wedge d\overline{\log \xi_j},
\]
which for $|\xi_i|\ll 1$ exponentially dominates the $T^m$-tori and the $m$ base directions, since $\sum d\log \xi_i \wedge d\overline{\log \xi_i}$ is exponentially larger than $\sum d\xi_i \wedge d\overline{ \xi_i}$ in the logarithmic coordinates. 
The term 
\[
-  \sum \frac{\partial u}{\partial x_i} dd^c \log r_i=  \sum \frac{\partial u}{\partial x_i} dd^c  \phi_i,
\]
since $dd^c \log |\xi_i|=0$ holds for $\xi_i\neq 0$, which is valid on the region $0<r_i\ll 1$ under consideration. 
The ansatz will only be used in the region with
\[
|D^2u |\ll 1\ll |Du|. 
\]
As such, we can essentially ignore the other contributions produced by the Hessian term.

There are compact directions diffeomorphic to $Y$. For this, notice for each fixed value of $x=(x_1, \ldots x_m)$, we get a $T^m$-invariant subset $Z_x\subset Z$, which projects down to $Y$ via the natural map $Z\to Y$. This projection map is topologically the quotient map by the $T^m$-action, so $Z_x/T^m$ is naturally diffeomorphic to $Y$. Since $Z_x/T^m$ has no a priori complex structure, one cannot say that this is a biholomorphism. Nevertheless, the natural Riemannian metric on $Z_x$ induces a metric on $Z_x/T^m$ by looking at the transverse directions to the $T^m$-action, and via the diffeomorphism this is close to the K\"ahler metric on $Y$
\begin{equation}\label{effectiveKahlerclass}
-dd^c \log h_L^{1/2} + \sum \frac{\partial u}{\partial x_i} dd^c  \phi_i.
\end{equation}
This expression requires some explanation: by definition $-dd^c \log h_L^{1/2}$ and  $dd^c  \phi_i$ make sense on $Y$. The term  $\frac{\partial u}{\partial x_i}$ is a function of $x$, and for fixed $x$ it is merely a constant coefficient on $Y$. Notice the metric (\ref{effectiveKahlerclass}) on $Y$  lies in the K\"ahler class $\mathcal{D}u=c_1(L)+ \sum  \frac{\partial u}{\partial x_i} c_1(L_i)$. The reason we are able to acquire a cohomology class which is not obvious in the topological setup, comes from the distributional terms of $dd^c \log |\xi_i|$ being discarded in the above calculations.  We will refer to this $\mathcal{D}u$ as an `effective K\"ahler class', since it is only present in the effective approximate description of the metric. Its size affects the length scale of the metric on the $Y$-fibres.

The next goal is to improve the metric so that the $Y$-fibres approximately have the Calabi-Yau metrics in the same class $c_1(L)+ \sum  \frac{\partial u}{\partial x_i} c_1(L_i)$. This is conceptually similar to the semi-Ricci-flat metrics in the context of holomorphic fibred Calabi-Yau manifolds \cite{Tosatti}. We take 
\[
h= h_1 \exp(- 2\phi_x),
\]
where for each $x$, we associate some potential $\phi_x$ on $Y$, which pulls back to $Z_x$, so varying over all $x$ we get a function on an open subset of $Z$. We shall assume that the dependence on $x$ is sufficiently weak. Then the dominant effect is to change the metric on $Z_x/T^m\simeq Y$ to 
\begin{equation}\label{effectivefibremetric}
-dd^c \log h_L^{1/2} + \sum \frac{\partial u}{\partial x_i} dd^c  \phi_i + dd^c \phi_x,
\end{equation}
and the other effects in the $\log \xi_i$ directions are suppressed. We choose $\phi_x$ so that the effective fibre metrics (\ref{effectivefibremetric}) are the unique Calabi-Yau metrics in the cohomology class $\mathcal{D}u(x)$. This determines $\phi_x$ up to fibrewise constants depending on $x$. For the purpose of constructing approximate Calabi-Yau metrics, the choice is inessential as long as the $x$-derivatives are small enough, just like what happens for semi-Ricci-flat metrics.

To leading order, 
\[
\begin{split}
& (-dd^c \log h^{1/2} )^n\approx \frac{n!}{(n-m)!} \det(D^2 u) \prod_1^m \frac{\sqrt{-1}}{4\pi} d\log \xi_i\wedge d\overline{\log \xi_i }
\\
& \wedge (-dd^c \log h_L^{1/2} + \sum \frac{\partial u}{\partial x_i} dd^c  \phi_i + dd^c \phi_x)^{n-m}
\\
=& \frac{n! \int_Y (\mathcal{D}u)^{n-m}}{(n-m)!} \det(D^2 u)
 (\prod_1^m \frac{\sqrt{-1}}{4\pi} d\log \xi_i\wedge d\overline{\log \xi_i } ) \wedge  \sqrt{-1}^{(n-m)^2} \Omega_Y\wedge \overline{\Omega}_Y,
\end{split}
\]
where $\Omega_Y$ is the holomorphic volume form on $Y$, normalized to \[
\sqrt{-1}^{(n-m)^2} \int_Y \Omega_Y\wedge \overline{\Omega}_Y=1.
\]
The term  $\int_Y (\mathcal{D}u)^{n-m}$ is intersection theoretic, and is polynomial in the derivative of $u$. 
The natural holomorphic form on $Z$ is up to a global constant
\begin{equation}\label{holovol}
\Omega= \prod_1^m d\log \xi_i \wedge \Omega_Y,
\end{equation}
which is well defined independent of trivializations. 
We see that
\[
(-dd^c \log h^{1/2} )^n\approx \text{const } \Omega \wedge \overline{\Omega}  \det(D^2 u)  \int_Y (\mathcal{D}u)^{n-m} .
\]
Thus provided the various assumptions involved in the approximation are satisfied, then the condition for the metric $-dd^c \log h^{1/2}$ to define an approximate Calabi-Yau metric is 
\begin{equation}\label{NAMA}
\det(D^2 u)  \int_Y (\mathcal{D}u)^{n-m} =\text{const}. 
\end{equation}
This is a Monge-Amp\`ere type PDE on the real $m$-dimensional base, which we call the \textbf{non-archimedean Monge-Amp\`ere equation} (NA MA for short) on account of a very similar construction in \cite{NASYZ}. The associated almost Calabi-Yau metric is called the \textbf{generalized Calabi ansatz}.

\subsection{Semiflat metrics and the Calabi ansatz}\label{semiflatCalabiansatz}

The generalized Calabi ansatz has two familiar special cases:

For $m=n$, then $Y$ is just a point, and $0<r_i\ll 1$ amounts to the subset $|\xi_i|\ll 1$ inside $(\C^*)^n$. The NA MA equation is just the real MA equation
\[
\det(D^2 u)=\text{const}. 
\]
The ansatz then produces a Calabi-Yau metric, known as the \textbf{semiflat metric}, familiar in the SYZ conjecture \cite{SYZ}.

For $m=1$, and $L$ trivial, then $Z$ is the total space of a positive line bundle $L_1$ over an $(n-1)$-dimensional Calabi-Yau manifold $Y$. We then take $h_{L_1}$ to be the Hermitian metric corresponding to the Calabi-Yau metric in $c_1(L_1)$. We can also view $h_{L_1}$ as a function on $Z$, equal to $r_{L_1}^2$. Up to a normalising constant, the \textbf{Calabi ansatz} is
\[
\omega_{Cal}= \frac{n}{n+1} dd^c (-\log h_{L_1}^{1/2})^{(n+1)/n}.
\]
We compare this to the NA MA equation (\ref{NAMA}) in this case: $u$ is a function of a single real variable $x$, satisfying
\[
u'' u'^{n-1}= \text{const}. 
\]
Up to constant $u=x^{(n+1)/n} $. Thus the NA MA equation reproduces the Calabi ansatz.

\subsection{Simplifications for proportional line bundles}
\label{proportionallinebundles}

A special case of the generalized Calabi ansatz is when $L$ is trivial, and $L_i=d_i L_0$ for some positive line bundle $L_0$ and positive integers $d_i>0$. In this case, we can take $h_{L}=1$, and $h_{L_i}$ is the suitable tensor power of $h_{L_0}$, where $h_{L_0}$ can be chosen to correspond to the Calabi-Yau metric in the class $c_1(L_0)$. The ansatz metric is simply
\[
\begin{split}
dd^c u= \sum \frac{\partial^2 u}{\partial x_i \partial x_j} d\log r_i \wedge d^c \log r_j+  \sum \frac{\partial u}{\partial x_i} dd^c \phi_i
\\
=   \sum \frac{\partial^2 u}{\partial x_i \partial x_j} d\log r_i \wedge d^c \log r_j+  (\sum \frac{\partial u}{\partial x_i}d_i )dd^c \phi_0   .
\end{split}
\]
Observe that for rank reasons
\[
(\sum \frac{\partial^2 u}{\partial x_i \partial x_j} d\log r_i \wedge d^c \log r_j)^{m+1}=0, \quad  (dd^c \phi_0)^{n-m+1}=0.
\]
We compute the volume form using binomial expansion
\begin{equation*}
(dd^cu)^n= \frac{n!}{(n-m)!} (\sum \frac{\partial u}{\partial x_i}d_i )^{n-m} \det(D^2u) (\prod_1^m d\log r_i\wedge d^c\log r_i) \wedge (dd^c \phi_0)^{n-m}.
\end{equation*}
Since $(dd^c \phi_0)^{n-m}$ already exhaust all base terms, we can replace $\log r_i$ by $\log |\xi_i|$, and obtain
\begin{equation*}
(dd^cu)^n = \frac{n!}{(n-m)!} (\sum \frac{\partial u}{\partial x_i}d_i )^{n-m} \det(D^2u) (\prod_1^m \frac{ \sqrt{-1}}{4\pi} d\log \xi_i\wedge d\overline{\log \xi_i}) \wedge (dd^c \phi_0)^{n-m}.
\end{equation*}
The Calabi-Yau condition on $dd^c\phi_0$ and the normalization on $\Omega_Y$ imply
\[
(dd^c \phi_0)^{n-m}= (\int_Y c_1(L_0)^{n-m} ) \sqrt{-1}^{(n-m)^2} \Omega_Y\wedge \overline{\Omega}_Y  .
\]
The conclusion is that 
\begin{lem}\label{generalizedCalabimodelisCY}
As long as the NA MA equation holds
\begin{equation*}
\det(D^2 u)  (\sum \frac{\partial u}{\partial x_i}d_i )^{n-m}
=\text{const},  
\end{equation*}
then the generalized Calabi ansatz $dd^cu$ is a Calabi-Yau metric, in this case of proportional line bundles. 
\end{lem}

\begin{rmk}
It is understood that $u$ is strictly convex, and $ \sum \frac{\partial u}{\partial x_i}d_i$ is positive. These two conditions guarantee the metric is positive definite.
\end{rmk}

\subsection{Relevance to the Tian-Yau problem}\label{TianYauproblem}

The relevance of the generalized Calabi ansatz to the Tian-Yau problem (\cf Question \ref{TianYauproblem}) is as follows. Let $\bar{X}$ be a smooth Fano manifold, and $D=\sum D_i$ be an anticanonical divisor, which we assume to be simple normal crossing, with only multiplicity one components. Let $X=\bar{X}\setminus D$, then $X$ has a nowhere vanishing holomorphic volume form $\Omega$, and one can ask when this admits a complete Calabi-Yau metric.

 An important intuition to keep in mind, is that most of the volume growth near the infinity of $X$ in fact concentrates near the deeper intersection strata $D_J=\cap_{i\in J} D_i$ of the component divisors $D_i$. Let $m$ denote the maximal $|J|$ for which  $D_J$ is non-empty, then from the volume growth perspective, the neighbourhood of the $m$-fold intersection loci are the generic regions in the Calabi-Yau $X$.  For any such subset $J$ with $|J|=m$, the neighbourhood of $D_J$ is essentially the total space of $\oplus_{j\in J} \mathcal{O}(D_j)|_{D_J}$, which corresponds to $Z$, and an application of the adjunction formula shows $D_J$ is a compact Calabi-Yau, which corresponds to $Y$. Up to inessential normalization constants, the holomorphic volume form is (\ref{holovol}) up to negligible errors. In many examples all the $\mathcal{O}(D_i)$ are ample. The possible relevance of $L$ comes from the prescribed global K\"ahler class on $X$. The picture we would like to advocate, for which our present paper is a very special case, is that one can find new Tian-Yau type metrics on $X$, whose behaviour in the generic region is modelled on the generalized Calabi ansatz. Morever, further details from this paper suggests the whole problem is inductive on $m$, in the sense that what happens in non-generic regions is related to the generalized Calabi ansatz with smaller $m$.

\subsection{ODE reduction}

While we believe the generalized Calabi ansatz has wide applicability, both in the Tian-Yau problem, and in the collapsing polarized degeneration problem as described in \cite{NASYZ}, solving the NA MA equation (\ref{NAMA}) is practically quite nontrivial for $m\geq 2$. We shall now specialize to the proportional line bundle case of section \ref{proportionallinebundles}, and further assume $m=2$. The NA MA equation becomes a PDE with two independent variables
\begin{equation}
\det(D^2 u) (d_1 \frac{\partial u}{\partial x_1} + d_2 \frac{\partial u}{\partial x_2} )^{n-2}=\text{const}. 
\end{equation}
Motivated by the Calabi ansatz, we wish to look for \emph{homogeneous} solutions. A preliminary dimensional analysis is useful:
\[
u\sim O( |x|^\alpha), \quad Du\sim O( |x|^{\alpha-1} ), \quad D^2u\sim  O(|x|^{\alpha-2}).
\]
Thus we want
\[
2(\alpha-2)+ (n-2)(\alpha-1)=0, \quad \alpha= \frac{ n+2}{n}.
\]
We try the ansatz
\begin{equation}
x_2= \frac{d_2}{d_1} tx_1, \quad u(x_1,x_2)= x_1^{ \frac{ n+2}{n} } v(t).
\end{equation}
Routine computation then reduces the NA MA equation to an ODE:

\begin{lem}
Under the homogeneous ansatz, the NA MA equation is equivalent to
\begin{equation}\label{ODE1}
(  vv''- \frac{ 2}{n+2} v'^2    ) (  \frac{n+2}{n} v+(1-t) v' )^{n-2} =\text{const}. 
\end{equation}
\end{lem}

\begin{proof}
We compute
\[
\frac{\partial u}{\partial x_1}= ( \frac{n+2}{n} v- tv') x_1^{2/n}, \quad \frac{\partial u}{\partial x_2}= \frac{d_1}{d_2} v' x_1^{2/n},
\]
and the second derivatives
\begin{equation}\label{secondderivatives}
\begin{cases}
& \frac{\partial^2 u}{\partial x_1^2}= x_1^{2/n-1} \{ \frac{2(n+2)}{n^2} v- \frac{4t}{n} v' +t^2v''    \},
\\
& \frac{\partial^2 u}{\partial x_1\partial x_2}= \frac{d_1}{d_2} (\frac{2}{n} v'- t v'') x_1^{\frac{2}{n}-1},
\\
& \frac{\partial^2 u}{\partial x_2^2}= (\frac{d_1}{d_2})^2 x_1^{\frac{2}{n}-1  }v''.
\end{cases}
\end{equation}
Whence
\[
\det(D^2 u)= (\frac{d_1}{d_2})^2 x_1^{ \frac{4}{n}-2 } (  \frac{2(n+2)}{n^2} vv''- \frac{4}{n^2} v'^2    ),
\]
\[
d_1 \frac{\partial u}{\partial x_1}+d_2 \frac{\partial u}{\partial x_2} 
= d_1 x_1^{2/n} (  \frac{n+2}{n} v+(1-t) v' ),
\]
so the NA MA equation becomes
\[
(  \frac{2(n+2)}{n^2} vv''- \frac{4}{n^2} v'^2    ) (  \frac{n+2}{n} v+(1-t) v' )^{n-2} =\text{const}. 
\]
\end{proof}

\begin{rmk}
The constant is not essential: it just amounts to rescaling the metric. 
\end{rmk}

\begin{rmk}\label{Kahlerconditionrmk1}
The K\"ahler condition requires $D^2u$ to be positive definite, and $d_1 \frac{\partial u}{\partial x_1} + d_2 \frac{\partial u}{\partial x_2}>0$. These are equivalent to
\begin{equation*}
v''>0, \quad
(n+2) vv''- 2v'^2 >0, \quad \frac{n+2}{n} v+(1-t) v'>0.
\end{equation*}
The first two inequalities imply $v>0$. These constraints are all quite natural in view of the ODE. 
\end{rmk}

\begin{rmk}\label{ODEsymmetryrmk1}
The ODE enjoys a symmetry: under the substitution
\[
v(t)= t^{\frac{n+2}{n }} \tilde{v}(1/t),
\]
we have
\[
(n+2) vv'' - 2v'^2= t^{ \frac{4}{n}-2  } (  (n+2) \tilde{v}'' \tilde{v}-2 \tilde{v}'^2),
\]
\[
\frac{n+2}{n} v+(1-t) v'= t^{2/n} \{    \frac{n+2}{n} \tilde{v} +(1-t^{-1} ) \tilde{v}'      \},
\]
so the function $\tilde{v}$ is another solution of the same ODE.  The geometric origin of this symmetry is that the NA MA equation is symmetric in $x_1, x_2$, up to the minor issue of $d_1, d_2$ which disappears after trivial changes of variables.
\end{rmk}

\subsection{Basic length scales of the generic region}\label{Basiclengthscales}

We mentioned in the beginning that the generalized Calabi ansatz geometrically describes an iterated fibration, which is supposedly the model for the generic region near infinity on the noncompact Calabi-Yau $X=\bar{X}\setminus D$. Figuring out the order of magnitude of various length scales is essentially a matter of dimensional analysis. In the generic region $t$ is of order $\sim1$, $v$ is smooth and of order $\sim 1$. Our homogeneous ansatz prescribes
\[
u\sim O(|x|^{(n+2)/n}), \quad |du|\sim O( |x|^{2/n}), \quad |D^2u|\sim O(|x|^{(2-n)/n}).
\]
We have $|D^2 u|\ll 1\ll |du|$.
From the descriptions in section \ref{generalizedCalabiansatz}, the Hessian term is responsible for the base and torus direction of the metric, while $du$ is responsible for the effective K\"ahler class, and therefore the size of the fibres diffeomorphic to $Y\simeq D_1\cap D_2$.
Thus
\[
\text{diam}(T^2)=O(|D^2u|^{1/2})= O(|x|^{(2-n)/2n}), \quad \text{diam}(Y)=O(|du|^{1/2})= O( |x|^{ 1/n  }),
\]
the distance to the origin is of order
\[
O(\int |D^2u|^{1/2} )=O( |x|^{(2+n)/2n}),
\]
and the volume within $|x|\leq r$ is $O(r^2)$ from the two log directions of the base. In terms of geodesic distance to the origin, the volume grows with power $4n/(n+2)$.

For $|x|\gg 1$, this reaffirms the intuition that the $T^2$ length scale is far smaller than $Y\simeq D_1\cap D_2$, which is far smaller than the real 2-dimensional base.

\subsection{Boundary condition of the ODE}\label{BoundaryconditionODE}

The ODE (\ref{ODE1}) is posed on $0<t<\infty$, and we now discuss the \textbf{boundary conditions} to put at $t=0$ and $t=\infty$. The infinity boundary can be reduced to the $t=0$ case by the symmetry of the ODE, which geometrically comes from the symmetry between $x_1, x_2$ (up to elementary scaling). Geometrically one would like to find a partial completion of the generalized Calabi ansatz, near the infinity of $\bar{X}\setminus D_1\cup D_2$. The generalized Calabi ansatz works in the neighbourhood of $D_1\cap D_2$, and one would like to understand what happens near $D_1\setminus D_2$ and $D_2\setminus D_1$.

The simplest boundary conditions to imagine is for the ODE to remain analytic at $t=0$. In this case, one would specify $v(0)=v_0>0$ and $v'(0)>- \frac{n+2}{n} v_0$, and the ODE is non-singular in a neighbourhood of $t=0$, so that $v$ has a Taylor expansion at $t=0$. Geometrically, the size of $T^2$ remains of order $O(x_1^{\frac{2-n}{2n}})$ and the effective K\"ahler class of $Y$ remains of order $O(x_1^{2/n})$ as $t\to 0$. There seems to be no known metric near $D_1\setminus D_2$ that one may attempt to glue to the generalized Calabi ansatz. Intuitively, one cannot suddenly `switch off' the effective K\"ahler class at $t=0$.

We are thus led to look for boundary conditions such that $\frac{n+2}{n} v+(1-t)v'\to 0$ as $t\to 0$, which intuitively means the effective K\"ahler class `switches off gradually'. The following type of boundary conditions is perhaps best motivated by fixing $v_0>0$, and trying to decrease $v'(0)$ until $v'(0)+ \frac{n+2}{n} v_0$ tends to zero. As $t\to 0$, we require $v'$ has some subleading power law behaviour
\begin{equation}\label{ODEboundary1}
\begin{cases}
& v=v_0+ O(t),
\\
& v'= - \frac{n+2}{n} v_0 + at^{\beta}+ O(t), 
\\
& v''= a\beta t^{\beta-1  }+O(1).
\end{cases}
\end{equation}
Here $0<\beta<1$ and $a>0$ are some constants to be determined. Thus
\[
vv''- \frac{2}{n+2} v'^2= av_0\beta t^{\beta-1} +O(1),
\]
\[
\frac{n+2}{n} v+ (1-t)v'= at^\beta+O(t),
\]
and the ODE predicts
\[
\beta-1+ (n-2)\beta=0,\quad \beta= \frac{1}{n-1},
\]
and
\[
\frac{av_0}{ n-1  } a^{n-2}= \text{const}, \quad a\propto v_0^{-1/(n-1)}.
\]

\subsubsection*{Geometric meaning of the boundary condition}

We can translate the boundary condition near $t=0$ into the asymptotic behaviour of the Calabi-Yau metric for $1\ll x_2\ll x_1$. Using the second derivative computation (\ref{secondderivatives}), we get
\[
\begin{cases}
\frac{\partial^2 u}{\partial x_1^2} \sim   \frac{2(n+2)}{n^2} v_0 x_1^{\frac{2}{n} -1},
\\
\frac{\partial^2 u}{\partial x_1 \partial x_2}\sim - \frac{d_1}{d_2}   \frac{2(n+2)}{ n^2 } v_0    x_1^{\frac{2}{n} -1}
\\
\frac{\partial^2 u}{\partial x_2^2} \sim (\frac{d_1}{d_2})^2  \frac{a}{n-1} t^{- \frac{ n-2}{n-1}}  x_1^{\frac{2}{n} -1}.
\end{cases}
\]
Recall that in the generalized Calabi-Yau ansatz, this Hessian matrix controls the base metric and the torus fibres. In the $x_1$-direction of the base (\ie the direction transverse to $D_1$ at infinity), and the $\log \xi_1$ circle direction, the behaviour is similar to what happens in the generic region where $t$ is of order one. However, in the $x_2$ direction, the base metric has a different scaling law:  
\[
\frac{\partial^2 u}{\partial x_2^2} \sim (\frac{d_1}{d_2})^2  \frac{a}{n-1} t^{-  \frac{n-2}{n-1}}  x_1^{\frac{2}{n} -1} =O( x_1^{\frac{2}{n}- \frac{1}{n-1}} x_2^{- \frac{n-2}{n-1} }).
\]
The distance to $x_2\sim O(1)$ is $O( x_1^{ \frac{1}{n}-\frac{1}{2(n-1)}} x_2^{ \frac{n}{2(n-1)}} )$. The $\log \xi_2$ circle direction now has diameter of order
\[
O(  x_2^{ - \frac{n-2}{2(n-1)} } x_1^{ \frac{1}{n}- \frac{1}{2(n-1)}   } ), 
\]
which is much larger compared to the $\log \xi_1$ circle. In other words, the metric exhibits \textbf{inhomogeneous collapsing} phenomenon.

%At a more quantitative level, we can examine the metric restricted to constant $\log x_1$ slices, which can be viewed as metrics on local regions inside the total space of $\mathcal{O}(D_2) \to D_1\cap D_2$. 

The potential is roughly
\begin{equation*}
\begin{split}
& u= x_1^{ \frac{n+2}{n}  } v(t)\approx x_1^{ \frac{n+2}{n}  } (v_0- \frac{n+2}{n} v_0 t + \frac{n-1}{n} a t^{ \frac{n}{n-1}  })
\\
& \approx v_0 (x_1- \frac{d_1}{d_2} x_2)^{ \frac{n+2}{n}  } + \frac{n-1}{n} ax_1^{ \frac{n+2}{n}  }  ( \frac{d_1}{d_2} \frac{x_2}{x_1}  )^{ \frac{n}{n-1}  }.
\end{split}
\end{equation*}
Here $x_2\ll x_1$. For fixed values of $x_1- \frac{d_1}{d_2} x_2$, the second term dominates the metric contribution on the slices.
Up to scaling factors by powers of $x_1$, the key dependence on $x_2$ is $dd^c x_2^{ n/(n-1)  }$, which is the \textbf{Calabi ansatz} on an open subset of the $(n-1)$-dimensional total space of the line bundle $\mathcal{O}(D_2)\to D_1\cap D_2$ (\cf section \ref{semiflatCalabiansatz}). When $\xi_1$ is allowed 
to vary, the scales of the $S^1$, the Calabi ansatz, and the base metric, all depend on $x_1$ in some power law fashion.

The significance to the compactification question, is that $\mathcal{O}(D_2)\to D_1\cap D_2$ describes the infinity of the non-compact Calabi-Yau manifold $D_1\setminus D_2$, and the Calabi ansatz is the asymptote of the Tian-Yau metric on $D_1\setminus D_2$. We can thus glue the Calabi ansatz to the Tian-Yau metric, in a parametrized fashion over the $x_1$ variable, in order to achieve the partial completion of the generalized Calabi ansatz. After this, we would have an approximate Calabi-Yau metric outside a compact set in $X=\bar{X}\setminus D_1\cup D_2$, from which one can hope to use a non-compact version of Yau's proof to the Calabi conjecture to obtain an actual Calabi-Yau metric on $X$.

\begin{rmk}
An appealing feature is that the boundary behaviour of the $m=2$ generalized Calabi ansatz is essentially the ordinary Calabi ansatz. We think this feature may generalize to larger $m$, so that the boundaries have an inductive stratification structure.

\end{rmk}

\subsection{Remarks on other related literature}

The main previously known source of complete Calabi-Yau metrics on the complement of a {\em singular} anticanonical divisor are due to Hein \cite{Heingravitational} on certain complex surfaces.  These examples are constructed on rational elliptic surfaces, which in particular admit elliptic fibrations onto $\mathbb{P}^1$ with fibers lying in the anticanonical linear system.  Hein constructs complete Calabi-Yau metrics asymptotic to the {\em semi-flat} Calabi-Yau metrics discovered by Greene-Shapere-Vafa-Yau \cite{GSVY}.  Interestingly, for Kodaira type $I_{b}$ singular fibers, these metrics turn out to be the same as the metrics constructed by Tian-Yau \cite{TianYau} on the complement of an ample anticanonical divisors in a del Pezzo surface \cite{CJY, CJY2, HSVZ}.  In what follows we collect some sporadic comparisons to other works in the literature which are not directly connected to the Tian-Yau problem.

\subsubsection{Degenerating hypersurfaces}

Sun and Zhang \cite{SunZhang} studied the degenerating Calabi-Yau metric on the family of hypersurfaces 
\[
X_t= \{  F_1 F_2+ tF=0 \}\subset \mathbb{CP}^{n}, \quad 0<|t|\ll 1,
\]
where $F_1, F_2$ define two transverse degree $d_1, d_2$ smooth irreducible hypersurfaces $D_1, D_2$, with $d_1+d_2=n+1$, and $F$ defines a generic hypersurface of degree $n+1$, so that the $F=0$ locus in $D_1\cap D_2$ is smooth and irreducible. A concrete special case, studied previously by \cite{HSVZ}, is when $X_t$ is a family of quartic K3 surfaces degenerating into the union of two quadrics.

The Calabi-Yau metric on $X_t$ is fibred over an interval. The ends of the interval correspond to the two regions $D_1\setminus D_2$ and $D_2\setminus D_1$, and the metrics therein are modelled on the Tian-Yau metrics, whose asymptotes match up with the Calabi ansatz.  The transition between the two Calabi ansatzs on the two ends is modelled on an Ooguri-Vafa type metric, obtained by a generalized Gibbons-Hawking type ansatz.

Now algebro-geometrically, our Tian-Yau type space $\mathbb{CP}^n\setminus D_1\cup D_2$ can be imagined as the limit of $\mathbb{CP}^n\setminus X_t  $ as $t\to 0$.  It is then natural (but somewhat na\"ive) to imagine taking the usual Tian-Yau metric construction on $\mathbb{CP}^n\setminus X_t$, and try to extract limits. It is then not surprising that the Tian-Yau metric on $D_1\setminus D_2$ and $D_1\setminus D_2$ should appear in the asymptotic description of the metric on $\mathbb{CP}^n\setminus D_1\cup D_2$, even though the precise scaling factors of these Tian-Yau regions do not seem to be predicted by this na\"ive limit. However, the Ooguri-Vafa type region in \cite{SunZhang} has no direct relation to the generalized Calabi ansatz in our construction.

There is a further way our construction is related to a natural generalization of \cite{SunZhang}:
\[
X_t= \{  F_1 F_2 F_3+ tF=0 \}\subset \mathbb{CP}^{n}, \quad 0<|t|\ll 1. 
\]
The algebro-geometric limit as $t\to 0$ is the union of transversely intersecting hypersurfaces $D_1, D_2, D_3$. One can similarly ask for the description of the Calabi-Yau metric on $X_t$ for small $t$. It is quite conceivable that the metric model in the region $D_1\setminus D_2\cup D_3$ (and the cyclic permutations) is provided by our construction, although how the transition happens between these three ends is an interesting open problem, which likely involves a further generalization of the Ooguri-Vafa type metric in \cite{SunZhang}.

As a more general remark, we think the higher $m$ version of the NA MA equation in this paper is the noncompact analogue of the NA MA equation appearing in the collapsing case of polarized degeneration of Calabi-Yau metrics explained in \cite{NASYZ}, and we expect the generalized Calabi ansatz to be relevant for local metric models in the polarized degenerations.

\subsubsection{Exotic metrics on $\C^n$}

There are a number of recent constructions of complete Calabi-Yau metrics on $\C^n$ with $ n\geq 3$, that share the unifying theme of \textbf{holomorphic fibrations}. The works \cite{LiC3}\cite{Gabor}\cite{Ronan} start with a holomorphic fibration given by a weighted homogeneous polynomial $F: \C^n\to \C$, such that $F^{-1}(0)$ carries a Sasakian-Einstein cone metric. The other fibres carry asymptotically conical Calabi-Yau metrics modelled at infinity on $F^{-1}(0)$. From this, one builds a Calabi-Yau metric on $\C^n$, whose fibrewise restrictions are approximated by these Calabi-Yau metrics on the fibres outside of a compact region, and in the horizontal direction is approximated by the pullback of the Euclidean metric on $\C$. Such metrics on $\C^n$ have \emph{maximal volume growth}, and the tangent cone at infinity is $F^{-1}(0)\times \C$ with the product metric. From an algebro-geometric perspective, the singularities on $F^{-1}(0)$ are klt, which should be viewed as mild singularities. The coordinate functions on $\C^n$ all have polynomial growth with respect to the geodesic distance to the origin, even though the growth rates are typically not linear.

\begin{rmk}\label{growthorder}
	One moral is that the holomorphic structure alone is very far from specifying the metric. The recent uniqueness result \cite{Gaboruniqueness} suggests that an additional filtration structure associated with the growth of holomorphic functions is key to the uniqueness and classification of the metrics.
\end{rmk}

More recently a family of new Taub-NUT type Calabi-Yau metrics were constructed on $\C^3$ \cite{LiTaubNUT}, using a generalized Gibbons-Hawking framework. The asymptotic geometry near infinity is generically a $T^2$-fibration over $\R^4$, but along three rays inside $\R^4$ the metric looks like a Taub-NUT fibration over a cylinder. These metrics are fundamentally different in that it has volume growth order $Vol(B(r))\sim O(r^4)$ coming from the $\R^4$ direction, which is not maximal volume growth. Algebro-geometrically, these metrics are associated with the holomorphic fibration
\[
F(z_1, z_2, z_3)= z_1z_2z_3: \C^3\to \C,
\]
whose fibres are generically cylinders, contributing two dimensions to $T^2$ and two other dimensions to $\R^4$. The $\C$ factor contributes the other two dimensions to $\R^4$. Notice the singular fibres are reducible, and the nature of the singularity is much worse than klt. In terms of the growth of holomorphic functions, only $z_1z_2z_3$ has polynomial growth, while $z_1, z_2, z_3$ individually all have exponential type growth, meaning that $\log |z_k|$ has polynomial growth.

\begin{rmk}
The Taub-NUT metric is a much more classical prototype. Its asymptotic geometry is an $S^1$-bundle over $\R^3$, with non-maximal volume growth $Vol(B(r))=O(r^3)$. Algebro-geometrically the Taub-NUT is associated with the fibration $\C^2\xrightarrow{z_1z_2}\C$, whose fibres are cylinders. The holomorphic function $z_1z_2$ has polynomial growth, while $z_1, z_2$ individually have exponential growth.
\end{rmk}

In view of these constructions, the new feature of this paper is that in the generalized Calabi ansatz, the approximately Calabi-Yau  fibres do not appear as fibres of holomorphic fibrations, but rather come from the effective description of an iterated non-holomorphic fibration. The algebraic functions on $X$ have exponential type growth.
The prototype of these phenomena is of course already known in the case of the Calabi ansatz, but we believe the NA MA equation points towards a much larger generality of examples, not limited to ODE reduction methods.

\begin{rmk}
A recent paper of Biquard and Delcroix \cite{BiquardDelcroix} constructs Calabi-Yau metrics on certain rank 2 complex symmetric spaces using small cohomogeneity methods. A real Monge-Amp\`ere type ODE \cite[Prop 2.3]{BiquardDelcroix} also plays a prominent role, and the relation with our construction seems to deserve some further investigation.
\end{rmk}

%\subsubsection{Rank 2 complex symmetric spaces}

\subsubsection{Non-archimedean meaning?}\label{nonarchimedeanmeaning}

As we mentioned before, the generalized Calabi ansatz was discovered in an attempt to interpret the non-archimedean version of the Monge-Amp\`ere equation in the context of polarized degenerations \cite{NASYZ}. This relation to non-archimedean geometry remains conjectural,
because at least some regularity is needed in order for the NA MA equation to admit a metric interpretation, which is unfortunately still not proven even in some cases where the Calabi-Yau metric is completely understood, such as the case studied in \cite{SunZhang}. The reader is thus warned that the following discussions will be rather speculative; they are meant to provide a more general and higher brow perspective to section \ref{TianYauproblem}, and to motivate directions of future research.

In the polarized degeneration setting, one associates dual complexes to SNC models (or more generally dlt models) of the degeneration family. When the SNC models are related by blow ups with centres supported on the central fibre, then there exist comparison maps between the dual complexes, and the Berkovich space is the inductive limit of these dual complexes. One should think of the Berkovich space as encoding the pure birational geometry of the degeneration family. The polarization provides an extra positive line bundle structure on the Berkovich space, which one should imagine as metric information. One can make sense of the non-archimedean version of plurisubharmonic functions and semipositive metrics, and associate the non-archimedean Monge-Amp\`ere measure. The foundational result of this non-archimedean pluripotential theory is that one can solve the non-archimedean Calabi conjecture on the Berkovich space by a variational method, as is expertly surveyed in \cite{Boucksom}.

Now in the noncompact setting, the natural analogue of SNC models is SNC pairs $(\bar{X},D)$ with $X=\bar{X}\setminus D$,\footnote{In general $\bar{X}$ needs not be Fano.} to which one can associate dual complexes and build up a version of the Berkovich space. To the author's knowledge, non-archimedean pluripotential theory has not been developed in this noncompact setting, but the key point we would like to suggest is that the non-archimedean Monge-Amp\`ere equation in this conjectural theory, should be equivalent to the differential geometric version (\ref{NAMA}), and its purpose is to prescribe the asymptotic of the K\"ahler potential.

The Berkovich space by itself only has the complex geometric information, and by Remark \ref{growthorder} we know that this is far from sufficient to specify the metric. Another problem with the non-compact setting, is that the Calabi-Yau volume is infinite. In the concrete setting of this paper, the problem of infinity is essentially solved by imposing the \emph{homogeneity ansatz}, which reduced the NA MA equation to a boundary value problem with two ends, which is morally a compact problem. Now the geometric meaning of the homogeneity ansatz has to do with the growth order of the algebraic functions with respect to the geodesic distance. This growth information goes beyond pure complex geometry, and knows something about the metric. We would like to suggest it plays a similar role to the positive line bundle in the context of polarized degenerations. Once this is taken into account, one can at least hope for a non-archimedean Calabi conjecture type result in this noncompact setting.

The above picture fits quite well with a heuristic principle of Yau, that  complete non-compact Calabi-Yau manifolds should (under mild conditions) admit a natural (quasi-)projective compactifications. When this compactification is projective, one may hope that under an additional specification of the homogeneity ansatz, the non-archimedean geometry produces a version of the NA MA solution, which one can then use to prescribe the asymptotic K\"ahler potential in the generic region. One then tries to find a completion of the ansatz metric, and hopes that a (highly elaborate) application of the Tian-Yau existence proof would eventually construct a Calabi-Yau metric.

\begin{rmk}
Currently it is an art to guess an appropriate homogeneity ansatz (\ie the growth order of the algebraic functions). Compatibility with the positivity requirements of K\"ahler geometry makes this a highly delicate issue. Could there be some connections to stability conditions?
\end{rmk}

This very large pool of potential examples still do not exhaust the full richness of the complete Calabi-Yau metrics. The reason is that in general Yau's compactification is only a \emph{partial compactification} into a  quasi-projective variety. A typical phenomenon is that there is a holomorphic fibration to a lower dimensional variety, and the partial compactification amounts to the compactification of the fibres. For instance, the Taub-NUT metric on $\C^2$ can be compactified into the rational surface
\[
\{  ( [X_0:X_1:X_2], y)| X_1X_2= yX_0^2         \}  \subset \mathbb{CP}^2\times \C_y,
\] 
where we added in two compactification divisors $D_1= \{ X_0=X_1=0  \}$ and $D_2= \{ X_0=X_2=0\}$. These divisors encode the exponential growth of the coordinate functions in the fibre direction, and are responsible for the fact that the fibrewise metric restrictions are approximately cylindrical. A very similar phenomenon happens with the Taub-NUT type metric on $\C^3$ mentioned above. The upshot is that by mixing the holomorphic fibration with the NA MA ansatz, one can hope to generate an even larger supply of Calabi-Yau metrics.

\section{More on the ODE reduction}\label{MoreonODE}

The aim of this section is to study further the ODE reduction (\ref{ODE1}) from the special case of the NA MA equation. It turns out the ODE can be solved exactly, and the solution is related to the hypergeometric function.

%In particular, we will show that the boundary condition (\ref{ODEboundary2}) indeed admits a unique fractional power series solution, and the solution can be extended at least to $t\leq 1$ including a small neighbourhood of the endpoint. We will then use an intermediate value theorem argument (relying partially on computer numerics) to show that the ODE admits a global solution on $0<t<+\infty$ with our boundary conditions, satisfying a certain symmetry with respect to $t\to t^{-1}$. 

\subsection{Reformulations of the ODE}

\subsubsection*{First reformulation of the ODE}

The ODE~(\ref{ODE1}) can be somewhat further simplified:

\begin{lem}
	Under the substitution $w= \frac{n+2}{n} v^{n/(n+2)}$, the ODE (\ref{ODE1}) is equivalent to
	\begin{equation}\label{ODE2}
	w''(  w+(1-t)w')^{n-2} = \text{const} \cdot w^{-3}.
	\end{equation}
\end{lem}

\begin{proof}
	Observe
	\[
	\left(\frac{v'}{v^{2/(n+2) }} \right)'= \frac{ vv''- \frac{2}{n+2} v'^2 }{ v^{ (n+4)/(n+2) } }.
	\]
	We can rewrite the ODE (\ref{ODE1}) as
	\[
	\frac{ vv''- \frac{2}{n+2} v'^2 }{ v^{ (n+4)/(n+2) } } \left(   \frac{n+2}{n} v^{n/(n+2)}  +(1-t) \frac{ v'}{v^{2/(n+2)} }               \right)^{n-2}= \text{const}\cdot  v^{-3n/(n+2)}.
	\]
	Now \[
	w= \frac{n+2}{n} v^{n/(n+2)} ,\quad w'= \frac{v'}{v^{2/(n+2)} }, \quad w''= \frac{ vv''- \frac{2}{n+2} v'^2 }{ v^{ (n+4)/(n+2) } }, 
	\]
	so the ODE simplifies to (\ref{ODE2}) after slightly modifying the constant. 
\end{proof}

\begin{rmk}\label{Kahlerconditionrmk2}
	The K\"ahler condition (\cf Remark \ref{Kahlerconditionrmk1}) translates into
	\begin{equation*}
	w>0,\quad w''>0, \quad w+(1-t)w'>0. 
	\end{equation*}
\end{rmk}

\begin{rmk}\label{ODEsymmetryrmk2}
	The symmetry of the ODE (\cf Remark \ref{ODEsymmetryrmk1}) translates into the following. Let
	\[
	w(t)= t \tilde{w}(1/t),
	\]
	then
	\[
	w'= \tilde{w}- t^{-1} \tilde{w}', \quad w''= t^{-3} \tilde{w}'', \]
	\[
	w+(1-t)w'= \tilde{w}+ (1-t^{-1}) \tilde{w}',\quad w'' w^3= \tilde{w}'' \tilde{w}^3. 
	\]
	Thus if $w$ solves (\ref{ODE2}), then so does $\tilde{w}$.
\end{rmk}

In terms of the substitution $w= \frac{n+2}{n} v^{n/(n+2)}$, the boundary condition at $t\to 0$ becomes
\begin{equation}\label{ODEboundary2}
\begin{cases}
& w= w_0+O(t),\quad w_0=  \frac{n+2}{n} v_0^{n/(n+2)},
\\
& w'= - w_0+ b_1 t^{1/(n-1)} +O(t), \quad b_1= av_0^{-2/(n+2)},
\\
& w''=  \frac{b_1}{n-1} t^{-(n-2)/(n-1)} +O(1). 
\end{cases}
\end{equation}
In the \textbf{normalization} of the ODE
\begin{equation}\label{ODE3}
w'' w^3( w+(1-t)w')^{n-2}= \frac{1}{n-1},
\end{equation}
we would have
\[
b_1^{n-1} w_0^3 =1. 
\]
Notice $w_0$ determines $b_1$, which means this boundary condition at $t=0$ comes in a 1-parameter family, instead of the generic 2-parameter family for second order ODEs. The $t=+\infty$ end has a closely related boundary condition via the ODE symmetry (\cf Remark \ref{ODEsymmetryrmk2}), which also arises in a 1-parameter family, so one expects the global solutions to the ODE to be isolated.

\subsubsection*{Second reformulation of the ODE}

We write for $0<t<1$,
\begin{equation}\label{wfrak}
s= \frac{1}{1-t} \in (1,\infty) , \quad \mathfrak{w}(s)= \frac{w(t)}{1-t}.
\end{equation}
Then
\[
\frac{ d\mathfrak{w}}{ds}= (1-t)^2 \frac{ d\mathfrak{w}}{dt} = (1-t)^2 \frac{ d}{dt} ( \frac{w}{1-t})= (1-t) w'+ w.
\]
\[
\frac{ d^2 \mathfrak{w}} {ds^2}= (1-t)^2 \frac{ d}{dt} ( (1-t) w'+ w) = (1-t)^3 w''. 
\]
Thus the ODE (\ref{ODE3}) can be reformulated as
\begin{equation}\label{ODE4}
\frac{d}{ds}( \frac{ d\mathfrak{w}}{ds}   )^{n-1}   = (n-1)\frac{ d^2 \mathfrak{w}} {ds^2} (\frac{ d\mathfrak{w}}{ds}   )^{n-2} = \mathfrak{w}^{-3}.
\end{equation}
The K\"ahler condition (\cf Remark \ref{Kahlerconditionrmk2}) translates into
\begin{equation}
\mathfrak{w}>0, \quad \frac{ d\mathfrak{w}}{ds} >0, \quad \frac{ d^2 \mathfrak{w}} {ds^2}>0.
\end{equation}
The initial condition at $s=1$ is
\begin{equation}\label{ODE4initial}
\mathfrak{w}= w_0 +O(t), \quad \frac{ d\mathfrak{w}}{ds}= b_1 t^{1/(n-1)} +O(t), \quad t= 1-s^{-1}.
\end{equation}

\subsection{First integral of the ODE}

\begin{lem}
If $\mathfrak{w}$ solves the ODE (\ref{ODE4}) with the initial condition (\ref{ODE4initial}), then
\begin{equation}\label{firstintegral}
\frac{n-1}{n} ( \frac{ d\mathfrak{w}}{ds})^{n} + \frac{1}{2 \mathfrak{w}^2} = \frac{1}{2w_0^2}.
\end{equation}
\end{lem}

\begin{proof}
We differentiate
\[
\begin{aligned}
\frac{d}{ds}\left(\frac{(n-1)}{n} \left(\frac{d\mathfrak{w}}{ds}\right)^{n} + \frac{1}{2\mathfrak{w}^2}\right) &= (n-1)\left(\frac{d\mathfrak{w}}{ds}\right)^{n-1}\left(\frac{d^2\mathfrak{w}}{ds^2}\right) - \mathfrak{w}^{-3}\frac{d\mathfrak{w}}{ds}\\
&=\left(\frac{d\mathfrak{w}}{ds}\right)\left(\frac{d}{ds}\left(\frac{d\mathfrak{w}}{ds}\right)^{n-1} - \mathfrak{w}^{-3}\right) \\
&=0
\end{aligned}
\]
Now the result follows by observing that, from the initial conditions we have
\[
\left(\frac{(n-1)}{n} \left(\frac{d\mathfrak{w}}{ds}\right)^{n} + \frac{1}{2\mathfrak{w}^2}\right)\bigg|_{s=1} = \frac{1}{2w_0^2}.
\]

\end{proof}

This first integral can be solved explicitly. 
Rewriting the first integral in terms of the rescaled variables $\frac{\mathfrak{w} }{w_0}$ and $\left( \frac{2(n-1)}{n} w_0^{n+2}  \right)^{-1/n} s$, the ODE simplifies to the form
\[
f'^n= 1- f^{-2}.
\]
We introduce the function
\[
F(x)= \int_1^x  (1- y^{-2})^{-1/n} dy,
\] 
then 
\begin{equation}\label{implicitsolution}
F(  \frac{\mathfrak{w} }{w_0} )= \left( \frac{2(n-1)}{n} w_0^{n+2}  \right)^{-1/n} (s-1).
\end{equation}
Inverting $F$ solves $\mathfrak{w}$ as a function of $s$. It is clear that 
\[
\mathfrak{w}>0, \quad \frac{ d\mathfrak{w}}{ds} >0.
\]
The ODE (\ref{ODE4}) then implies $\frac{ d^2\mathfrak{w}}{ds^2} >0$, namely the K\"ahler condition is satisfied. We remark that $F$ can be expressed in terms of the hypergeometric functions, \cf the Appendix.

We then check the initial conditions and the analyticity of the solution near $s=1$.

\begin{cor}\label{ODEanalyticity}
The function $\mathfrak{w}(s)$ is a power series in $(s-1)^{n/(n-1)}$ near $s=1$. To leading orders
\begin{equation}\label{ODEasymptoterefined}
\mathfrak{w}=w_0+ \frac{n-1}{n} w_0 ^{- \frac{3}{n-1}}(s-1)^{\frac{n}{n-1}} +O((s-1)^{ \frac{2n}{n-1}  }  ).
\end{equation}
\end{cor}

\begin{proof}
Notice near $y=1$, the function  
\[
(1- y^{-2})^{-1/n}= (y-1)^{-1/n}\times \text{Taylor series in $y-1$ with constant term $2^{-1/n}$}.
\]
Upon integration,
\[
F(x)= (x-1)^{(n-1)/n}\times \text{Taylor series in $x-1$  with constant term $2^{-1/n}\frac{n}{n-1}$}.
\]
Raising (\ref{implicitsolution}) to the power $n/(n-1)$, we see 
\[
 w_0 ^{- \frac{n+2}{n-1}}(s-1)^{n/(n-1)}= \text{Taylor series in $ \frac{\mathfrak{w} }{w_0}  -1$ with first coefficient $\frac{n}{n-1}$}.
\]
Inverting the function, $ \frac{\mathfrak{w} }{w_0}  -1$ is a power series of $(s-1)^{n/(n-1)}$ near $s=1$. To leading order,
\[
\frac{\mathfrak{w}}{w_0}-1= \frac{n-1}{n} w_0 ^{- \frac{n+2}{n-1}}(s-1)^{\frac{n}{n-1}}\left( 1+O( (s-1)^{\frac{n}{2n-1}}  )\right), 
\]
\[
\frac{ d\mathfrak{w}}{ds} \approx w_0 ^{- \frac{3}{n-1}}(s-1)^{1/(n-1)}= b_1  (s-1)^{1/(n-1)},
\]
which agrees with the initial condition (\ref{ODE4initial}).
\end{proof}

\subsection{Global matching problem}

We have solved the ODE with the prescribed initial condition, on the interval $0<t<1$. For the application to the generalized Calabi ansatz, we need solutions over $0<t<\infty$, and for this purpose the ODE (\ref{ODE4}) is inadequate. From the ODE (\ref{ODE3}), it is however a priori clear that $t=1$ is not a singularity.

\begin{lem}
For any given $w_0>0$, the solution to the ODE (\ref{ODE3}) exists smoothly on $0<t<1+\epsilon(w_0)$ for some $\epsilon(w_0)>0$. 
\end{lem}

\begin{proof}
The ODE (\ref{ODE3}) can be smoothly extended as long as $w$ remains bounded positively below and $(1-t)w'+w$ remains bounded  (which imply boundedness of $w''$, and in particular the boundedness of $w, w'$). Notice
\[
\frac{d}{dt}((1-t)w'+w)= (1-t)w''>0,
\]
so $(1-t)w'+w$ is monotone increasing, and in particular positive. By
\[
\frac{d}{dt} ((1-t)w'+w)^{n-1} = \frac{1-t}{w^3}, 
\]
we see $(1-t)w'+w$ will be bounded as long as $w$ is bounded positively below.

 The convexity of $w$ is ensured whenever the solution is smooth. Thus for some small $\epsilon>0$,
\[
w(t)\geq w(\epsilon)- w'(\epsilon)(t-\epsilon), \quad t\geq \epsilon.
\]
For small $\epsilon$, we have $w_0-w_0\epsilon <w(\epsilon)<w_0$ and $-w_0< w'(\epsilon)<0$, so $w(\epsilon)/w'(\epsilon)> 1-\epsilon$, whence $w(t)$ has an a priori lower bound slightly beyond $t=1$.
\end{proof}

Recall the ODE has a symmetry under $t\to t^{-1}$ (\cf Remark \ref{ODEsymmetryrmk2}). Our strategy to achieve both the $t=0$ and the $t=\infty$ boundary conditions, is to look for symmetric solutions:
\[
w(t)= t\tilde{w}(1/t), \quad w(t)=\tilde{w}(t).
\]
This is a functional equation on $w(t)$, and it amounts to a \textbf{matching condition} at $t=1$:
\[
w(1)=\tilde{w}(1), \quad w'(1)= \tilde{w}'(1).
\]
This is equivalent to
\begin{equation}\label{matching}
w'(1)= \frac{1}{2} w(1).
\end{equation}
The problem is then to look for $w_0>0$ to solve (\ref{matching}). The key is to extract $w'(1)$ from the asymptote of $\mathfrak{w}$ as $s\to +\infty$, namely $t\to 1$. The starting point is the identity
\begin{equation}\label{matchingLegendre}
w'(t)= s \frac{d\mathfrak{w}}{ds}- \mathfrak{w}.
\end{equation}
which means the value of $w'$ is related to the Legendre transform of $\mathfrak{w}$.

\subsection{Legendre transform}

The \textbf{Legendre transform} of the convex function $\mathfrak{w}$ is given by
\[
\mathfrak{w}^*(p) = \sup_{s\in[1,\infty)} sp-\mathfrak{w}(s).
\]
In particular, we have
\[
\mathfrak{w}^*(p) = s\frac{d\mathfrak{w}}{ds}-\mathfrak{w}(s) \qquad \text{ at } p= \frac{d\mathfrak{w}}{ds}
\]

We are now going to rewrite the first integral (\ref{firstintegral}) in terms of $\mathfrak{w}^*(p)$.  First, since $\lim_{s\rightarrow \infty} \mathfrak{w}(s)= \infty$ we see from (\ref{firstintegral}) that
\[
\frac{d\mathfrak{w}}{ds}([1,\infty)) = [0, p_*=\left(\frac{n}{2(n-1)w_0^2}\right)^{\frac{1}{n}})
\]
and so the Legendre transform is defined on this interval.  Furthermore, since $\frac{d\mathfrak{w}}{ds}(s=1)=0$ we have $\frac{d\mathfrak{w}^*}{dp}(p=0)=1$, and $\mathfrak{w}^*(p=0)=-\mathfrak{w}(s=1)=-w_0$ by the properties of the Legendre transform.  Now from the involution property of the Legendre transform,
\[
\mathfrak{w}(s(p)) = s(p)p-\mathfrak{w}^*(p) = \frac{d\mathfrak{w}^*}{dp}p-\mathfrak{w}^*(p).
\]
The first integral (\ref{firstintegral}) is rewritten as
\[
\frac{n-1}{n}p^n + \frac{1}{2}\left(\frac{d\mathfrak{w}^*}{dp}p-\mathfrak{w}^*(p)\right)^{-2} = \frac{1}{2w_0^2},
\]
namely
\[
(p \frac{d}{dp}\mathfrak{w}^{*}-\mathfrak{w}^*)^2\left(\frac{1}{w_0^2}-\frac{2(n-1)p^n}{n}\right)= 1.
\]
To simplify matters we can rescale.  Define
\[
y= \frac{1}{p_*}p, \quad  g(y)= \frac{1}{w_0}\mathfrak{w}^*.
\]
Then the ODE is recast on the interval $y\in [0,1)$ as
\begin{equation}\label{ODE5}
(y\frac{dg}{dy}-g(y))^2(1-y^n) = 1,
\end{equation}
subject to the initial conditions
\[
g(0)=-1,\quad  g'(0) = \frac{p_*}{w_0}= \left(\frac{n}{2(n-1)w_0^{2+n}}\right)^{\frac{1}{n}}.
\]
This can be integrated explicitly:

\begin{lem}
The function 
\[
g(y)= - \,_{2}F_{1}[\frac{1}{2}, -\frac{1}{n}, \frac{n-1}{n}; y^n] +\frac{p_*}{w_0}  y, 
\]
where 
$\,_{2}F_{1}[a,b,c;z]$ is the hypergeometric function (\ref{Hypergeometric}). In particular,
\[
g(1)= \frac{p_*}{w_0}  - \frac{\Gamma(1-\frac{1}{n})\sqrt{\pi}}{\Gamma(\frac{1}{2}-\frac{1}{n})}.
\]
\end{lem}
\begin{proof}
	First note that if $g$ solves the equation, then so does $g+cy$ for any constant $c$.  Thus, we may reduce to the
	initial condition $g(0)=-1, g'(0)=0$ below. The initial condition specifies a sign choice of the square root, whence
	\[
	\frac{d}{dy}\left(\frac{g}{y}\right) = \frac{1}{y^2(1-y^n)^{\frac{1}{2}}},
	\]
and the Lemma is reduced to Prop. \ref{lem: LegendreODE}.
\end{proof}

We can implement the \textbf{matching condition} $w'(1)=\frac{1}{2}w(1)$ as follows. From (\ref{matchingLegendre}) and the definition of the Legendre transform, 
\[
w'(t=1) =  \mathfrak{w}^*(p_*) = w_0g(1)= w_0\left(\frac{p_*}{w_0}  - \frac{\Gamma(1-\frac{1}{n})\sqrt{\pi}}{\Gamma(\frac{1}{2}-\frac{1}{n})}\right).
\]
On the other hand, we have 
\[
\begin{aligned}
w(t=1) &= \lim_{s\rightarrow \infty}\frac{\mathfrak{w}(s)}{s}\\
& = \lim_{p\rightarrow p_*} \frac{p(d\mathfrak{w}^*/dp)-\mathfrak{w}^*}{(d\mathfrak{w}^*/dp)}\\
&= \lim_{p\rightarrow p_*} p-\frac{\mathfrak{w}^*}{(d\mathfrak{w}^*/dp)}
\end{aligned}
\]
but from the first integral we have $(d\mathfrak{w}^*/dp) \rightarrow \infty$ as $p\rightarrow p_*$, while $\mathfrak{w}^* \rightarrow \mathfrak{w}^*(1)<+\infty$. Thus we get
\[
w(1) =p_* = \left(\frac{n}{2(n-1)w_0^2}\right)^{1/n}.
\]
 All together, the matching amounts to solving the equation
\[
w_0\left( \frac{p_*}{w_0}-\frac{\Gamma(1-\frac{1}{n})\sqrt{\pi}}{\Gamma(\frac{1}{2}-\frac{1}{n})}\right)= \frac{1}{2}p_*,
\]
Or equivalently
\[
\frac{1}{2}\left(\frac{n}{2(n-1)w_0^2}\right)^{1/n}=w_0\frac{\Gamma(1-\frac{1}{n})\sqrt{\pi}}{\Gamma(\frac{1}{2}-\frac{1}{n})}
\]
which yields
\begin{equation}
\begin{aligned}
w_0
=\left( \frac{1}{2}\right)^{\frac{n+1}{n+2}} \left(\frac{n}{(n-1)}\right)^{\frac{1}{n+2}} \left(\frac{1}{\sqrt{\pi}}\frac{\Gamma(\frac{1}{2}-\frac{1}{n})}{\Gamma(1-\frac{1}{n})}\right)^{\frac{n}{n+2}}
\end{aligned}
\end{equation}
which tends to $\frac{1}{2}$ as $n \rightarrow +\infty$. For our purpose the important fact is that $w_0>0$ for $n\geq 3$, which means we have found the initial condition that ensures matching.

\begin{rmk}
If $n=2$, then $g(1)= \frac{p_*}{w_0}$, so $w'(t=1)= p_*$, and $w(t=1)= p_*$. The matching condition has no positive solution. Correspondingly, our generalization of the Calabi ansatz only yield nontrivial examples in dimension at least three. 
\end{rmk}

\section{Asymptotic ansatz and initial error}\label{initialerror}

The goal of this section is to produce a metric ansatz, and the main technical part is to estimate its deviation from being Calabi-Yau in weighted H\"older spaces. Our setting is a special case of section \ref{TianYauproblem}.  Let $n\geq 3$, and $\bar{X}$ be a Fano manifold of dimension $n$. The anticanonical bundle is $ (d_1+d_2)L_0$ for some positive line bundle $L_0$ over $\bar{X}$, and let $D_1, D_2$ be smooth divisors in the linear system $d_1L_0$ and $d_2L_0$ respectively, such that the intersection $Y=D_1\cap D_2$ is transverse. Lefschetz hyperplane theorem then implies that $D_1\cap D_2$ is smooth and irreducible (notice irreducibility would fail in dimension two, which is excluded in our construction). By adjunction, $Y$ is a compact Calabi-Yau manifold, and the noncompact manifold $X=\bar{X}\setminus D_1\cup D_2$ carries a natural nowhere vanishing holomorphic volume form.

\subsection{Generic region near infinity}\label{Genericregiongluingansatz}

We start from the solution $w(t)$ of the ODE (\ref{ODE3}) with the matching condition $w'(1)= \frac{1}{2}w(1)$, which guarantees the existence of the solution for $0<t<+\infty$. As discussed in section \ref{MoreonODE}, this specifies the parameter choices in (\ref{ODEboundary2})
\[
w_0= \left( \frac{1}{2}\right)^{\frac{n+1}{n+2}} \left(\frac{n}{(n-1)}\right)^{\frac{1}{n+2}} \left(\frac{1}{\sqrt{\pi}}\frac{\Gamma(\frac{1}{2}-\frac{1}{n})}{\Gamma(1-\frac{1}{n})}\right)^{\frac{n}{n+2}}, \quad b_1=  w_0^{-\frac{3}{n-1} }.
\]
The boundary behaviour near $t\to 0$ is prescribed by (\ref{ODEboundary2}). The $t\to +\infty$ boundary is determined from $w(t)=t w(1/t)$. Since the geometry of the $t\to +\infty$ limit is essentially identical to $t\to 0$, we will later only focus on $t\to 0$.

Reversing the previous ODE reductions, let 
\[
v(t)= (\frac{nw}{n+2})^{\frac{n+2}{n}}, \quad u(x_1,x_2)=  x_1^{\frac{n+2}{n} } v(t),\quad t= \frac{d_1x_2}{d_2x_1}.
\]
Then after restoring a few unpleasant constants
\[
(vv''- \frac{2}{n+2} v'^2)(\frac{n+2}{n} v+(1-t) v'  )^{n-2}= \frac{1}{n-1} (\frac{n}{n+2})^{3},
\]
\[
\det(D^2 u) (d_1 \frac{\partial u}{\partial x_1} + d_2 \frac{\partial u}{\partial x_2})^{n-2}=\frac{2n}{(n+2)^2(n-1)}  (\frac{d_1}{d_2})^{2} d_1^{n-2}.
\]
In the asymptotic formula (\ref{ODEboundary1}) for $v$ at $t\to 0$, we can recover the constants
\begin{equation}\label{v0a}
v_0=  (\frac{nw_0}{n+2})^{\frac{n+2}{n}} , \quad a= b_1 v_0^{\frac{2}{n+2} }= (\frac{n}{n+2})^{2/n} w_0^{- \frac{n+2}{n(n-1)} }.
\end{equation}

Let $S_1\in H^0(\bar{X}, d_1L_0)$ (resp. $S_2\in H^0(\bar{X}, d_2L_0)$) be the defining section of $D_1\subset \bar{X}$ (resp. $D_2$). Recall $h_{L_0}$ is the Hermitian metric on the line bundle $L_0$ over $Y$, whose curvature form is the Calabi-Yau metric on $Y$ in the class $c_1(L_0)$. We extend $h_{L_0}$ smoothly over $\bar{X}$.  This induces Hermitian metrics on $d_1L_0$ and $d_2L_0$, and in particular we can make sense of
$
x_i= -\log |S_i|.
$
A technical subtlety is that $S_1, S_2$ have the ambiguity of a multiplicative constant, which corresponds to the ambiguity of additive constants on $x_1, x_2$ of order $O(1)$, to be fixed in section \ref{TianYaubackground}. 
Very large $x_1, x_2$ corresponds to the generic region near infinity on the noncompact manifold $X$. The function $u(x_1, x_2)$ can be regarded as a K\"ahler potential on the tubular neighbourhood of $Y$ (with $Y$ deleted). Up to exponentially small error in the $x_1, x_2$ variables, the K\"ahler metric is modelled on the generalized Calabi ansatz.

The normal bundle of $Y=D_1\cap D_2\subset \bar{X}$ is $\mathcal{O}(D_1)\oplus \mathcal{O}(D_2)|_Y= d_1L_0\oplus d_2L_0|_Y$.
Using an auxiliary smooth Hermitian metric on $\bar{X}$, we can identify this normal bundle with a tubular neighbourhood of $Y\subset \bar{X}$ (eg. via normal geodesic flow). The choices of the identifications only produce errors which are exponentially small in the logarithmic variables, which will be negligible since the distance scales of the ansatz metric have power law dependence on the log variables (\cf section \ref{Basiclengthscales}).

%Recall $h_{L_0}$ is the Hermitian metric on the line bundle $L_0$ over $Y$, whose curvature form is the Calabi-Yau metric on $Y$ in the class $c_1(L_0)$. This induces Hermitian metrics on the line bundles $d_1L_0$ and $d_2L_0$, hence radius distance functions $r_1, r_2$ on the total space of $d_1L_0$ and $d_2L_0$ as in section \ref{generalizedCalabiansatz}, regarded as functions on the tubular neighbourhood of $Y$. We write $x_i= -\log r_i$ for $i=1,2$. Very large $x_1, x_2$ means that we are close to $Y\subset \bar{X}$, and on the noncompact manifold $X=\bar{X}\setminus D_1\cup D_2$,  this is the generic region near infinity. The function $u(x_1, x_2)$ can be regarded as a K\"ahler potential on the tubular neighbourhood of $Y$ (with $Y$ deleted). Up to exponentially small error in the $x_1, x_2$ variables, the K\"ahler metric is modelled on the generalized Calabi ansatz.

Setting up the H\"older norms require a little care, since the injectivity radius of the $T^2$-fibres tends to zero in the generic region near infinity. However, for $1\ll x_2\leq x_1$ (the case of $1\ll x_1\leq x_2$ being entirely similar), the harmonic radius grows like $O(x_1^{1/n}t^{\frac{1}{2(n-1)} })= O(x_1^{ \frac{1}{n}- \frac{1}{2(n-1)}  } x_2^{ \frac{1}{2(n-1)}  }) $. This motivates the weighting function
\begin{equation}
\rho= \begin{cases}
 (|x_1|+1)^{  \frac{n-2}{2n(n-1)}  } (|x_2|+1)^{ \frac{1}{2(n-1)}  },\quad  x_2\leq x_1,
 \\
 (|x_2|+1)^{ \frac{n-2}{2n(n-1)}  } (|x_1|+1)^{ \frac{1}{2(n-1)}  }, \quad x_1\leq x_2.
\end{cases}
\end{equation}
We shall only use $\rho$ up to a uniform equivalence constant; no derivative control on $\rho$ is required.
Each $Y$-fibre is covered by $O(1)$ number of charts each of length scale $O(\rho)$, such that the $T^2$-bundle is trivialized over the charts.  Given a tensor field $T$ on an $O(\rho)$ neighbourhood inside $X$, we can pass to the \emph{local universal cover} of the charts by unwrapping the $T^2$ factors. The local H\"older seminorm is
\begin{equation}\label{localHolderseminorm}
[T]_{\alpha}:= \sup_{dist(P,Q)\lesssim \rho} \rho^\alpha \frac{ |T(P)-T(Q)|}{ |P-Q|^{\alpha}  } ,
\end{equation}
using geodesic parallel transport with respect to the generalized Calabi ansatz metric on the local universal cover of the charts. The local $C^{k,\alpha}$ norm is 
\begin{equation}\label{localHolder}
\norm{T}_{k,\alpha,loc}:= \sum_{j=0}^k \sup_{dist(P,Q)\lesssim \rho} \rho^j |\nabla^j T|  + \rho^k [\nabla^k T]_{\alpha}.
\end{equation}

There is up to constant multiple a natural holomorphic volume form $\Omega$ on $X=\bar{X}\setminus D_1\cup D_2$. 
We take the normalization such that near $Y\subset \bar{X}$ (\cf (\ref{holovol}) for the model case)
\[
\Omega= (1+ f_\Omega) \prod_1^2 d\log \xi_i \wedge \Omega_Y,
\]
where $\xi_1, \xi_2$ are local defining functions of the smooth divisors $D_1, D_2$, and $f_\Omega$ is a local holomorphic function near $Y$, of order $O(|\xi_1|+|\xi_2|)$, which is exponentially small in the $x_1, x_2$ variables. Higher order derivatives of $f_\Omega$ are also exponentially small by holomorphicity.

We write
\begin{equation}\label{K0}
(dd^c u)^n =K_0 (1+ Err_1) \sqrt{-1}^{n^2} \Omega\wedge \overline{\Omega},
\end{equation}
where $Err_1$ is some error function, and $K_0$ is the (unfortunately complicated) proportionality constant in the model case of generalized Calabi ansatz
\[
K_0= \frac{\int_Y c_1(L_0)^{n-2}}{ (4\pi)^2}\frac{2n^2}{(n+2)^2}  (\frac{d_1}{d_2})^{2} d_1^{n-2}.
\]
Recall that Lemma \ref{generalizedCalabimodelisCY} says the model case is exactly Calabi-Yau.

\begin{lem}\label{volumeformerror1}(Volume form error)
In the generic region $\min(x_1, x_2)\gg 1$, we have exponential decay on the local $C^{k,\alpha}$ norm of the error function:
$\norm{Err_1}_{k,\alpha, loc}=O( e^{- c\min(x_1, x_2) })$ for some $c>0$.

\end{lem}

\subsection{Asymptotic expansion near $D_1$}\label{AsymptoticexpansionnearD1}

We now translate the ODE asymptote at $t\to 0$ to the geometry of the region $1\ll x_2\ll x_1$. The analyticity of the ODE solution (\cf Cor. \ref{ODEanalyticity}) gives a fractional power series expansion (\ref{ODEasymptoterefined}), which converts via (\ref{wfrak}) to
\[
w(t)=(1-t)\{  w_0+ \frac{n-1}{n}w_0^{-\frac{3}{n-1} }(\frac{t}{1-t})^{\frac{n}{n-1} }  + O((\frac{t}{1-t})^{\frac{2n}{n-1} }  )\}.
\]
Now
\[
w= \frac{n+2}{n}v^{n/(n+2)}, \quad v_0= (\frac{nw_0}{n+2})^{\frac{n+2}{n} }, \quad a= (\frac{n}{n+2})^{2/n} w_0^{ - \frac{n+2}{n(n-1)}  },
\]
hence we have a fractional power series
\[
\begin{split}
v(t)= (1-t)^{ \frac{n+2}{n}  }
\{ v_0   +\frac{n-1}{n} a  (\frac{t}{1-t})^{\frac{n}{n-1}  } + \sum_{k=2}^\infty a_k( \frac{t}{1-t} )^{\frac{kn}{n-1}  }   \}.
\end{split} 
\]

Recall
\[
x_2= \frac{td_2}{d_1} x_1, \quad u= x_1^{\frac{n+2}{n} }v(t).
\]
We introduce the new variable
\[
\tilde{x}_1:= x_1-\frac{d_1}{d_2}x_2= (1-t) x_1= \frac{d_1}{d_2} \frac{1-t}{t} x_2.
\]
In terms of the defining sections $S_i\in  H^0(\bar{X}, d_iL_0)$ for $D_1, D_2$, we have
\[
\tilde{x}_1= -\log |S_1|+ \frac{d_1}{d_2} \log |S_2|= \frac{1}{d_2} \log |S_1^{\otimes d_1}/S_1^{\otimes d_2}|,
\]
where 
\[
\xi':=S_2^{\otimes d_1}/S_1^{\otimes d_2}
\]
is actually a holomorphic \emph{function}, which means its magnitude does not involve the choice of a Hermitian metric. The geometric significance of this function is that although the normal bundle $\mathcal{O}(D_1)=d_1L_0$ of $D_1\subset\bar{X}$ is nontrivial, it restricts to a line bundle on $D_1\setminus D_2$ which becomes trivial after taking finite power.

We then have an expansion for $1\ll x_2\ll x_1$,
\begin{equation}\label{upowerseries}
u= v_0 \tilde{x}_1^{ \frac{n+2}{n} } +\frac{n-1}{n} a \tilde{x}_1^{ \frac{n+2}{n} -\frac{n}{n-1} } (\frac{ d_1x_2}{d_2})^{\frac{n}{n-1}  }+ \sum_{k\geq 2}
a_k \tilde{x}_1^{ \frac{n+2}{n} -\frac{kn}{n-1} } (\frac{ d_1x_2}{d_2})^{\frac{kn}{n-1}  }
 .
\end{equation}

\subsection{More background on the Tian-Yau metric}\label{TianYaubackground}

We recall some details of the \textbf{Tian-Yau metric} (\cf \cite[section 3]{HSVZ} for more expositions).

In our setup, $D_1$ is a Fano manifold in its own right with anticanonical bundle $d_2L_0|_{D_1}$, and $Y=D_2\cap D_1$ is an anticanonical divisor defined by some section $S\in H^0(D_1, d_2L_0)$, which in our context is the restriction of $S_2\in H^0(\bar{X}, d_2L_0)$ to $D_1$. Up to a multiplicative constant $S^{-1}$ can be viewed as a holomorphic volume form $\Omega_{D_1}$ on $D_1\setminus D_2$ with a simple pole along $Y$, which we normalize to have residue $\Omega_Y$ along $Y$. In local coordinates,
\[
\Omega_{D_1}\approx d\log \xi_2 \wedge \Omega_Y,
\]
where $\xi_2$ is a local defining function of $D_2\cap D_1\subset D_1$. Recall $h_{L_0}^{\otimes d_2}$ is the Hermitian metric on $L_0|_Y$ whose curvature form is the Calabi-Yau metric on $Y$ in the class $d_2c_1(L_0)$. We extend $h_{L_0}$ to a smooth, positively curved metric on $D_1$, so that
\[
\omega_{Cal'}= \frac{n-1}{n} dd^c (-\log |S|)^{ \frac{n}{n-1} }
\]
defines a K\"ahler form on a neighbourhood of infinity in $D_1\setminus D_2$. Up to an approximately holomorphic diffeomorphism, the metric $\omega_{Cal'}$ on $(D_1\setminus D_2)$ outside a compact set, agrees with the Calabi ansatz on the total space of the line bundle $L_0\to Y$, up to exponentially small errors. The slightly unusual exponent $\frac{n}{n-1}$ is because $\dim Y=n-2$. Asymptotically, the distance to a fixed basepoint in $D_1\setminus D_2$ is uniformly equivalent to
$
r_{D_1}\sim (-\log |S|)^{\frac{n  }{2(n-1) } }.
$
We can arrange $\omega_{Cal'}$ to extend to an exact global K\"ahler metric on $D_1\setminus D_2$, agreeing with the previous formula outside a compact set, still denoted $\omega_{Cal'}$.

A technical subtlety is that $S$ is only defined up to a multiplicative constant, and correspondingly $\log |S|$ has an additive constant ambiguity, which is fixed by the \textbf{integral normalization} condition
\begin{equation}\label{integralnormalization}
\int_{D_1\setminus D_2} \left(\omega_{Cal'}^{n-1}- \frac{d_2^{n-2}\int_Y c_1(L_0)^{n-2}}{4\pi} \sqrt{-1}^{(n-1)^2} \Omega_{D_1}\wedge \overline{\Omega}_{D_1}\right) =0.
\end{equation}
This makes sense because the integrand has exponential decay near infinity, and morever Stokes theorem on large compact sets shows that the integral only depends on the asymptotic information of $\omega_{cal'}$ near infinity. The preferred normalization of $\log |S|$ resolves the additive ambiguity of $x_2=-\log |S_2|$; a completely analogous normalization on the Tian-Yau space $D_2\setminus D_1$ fixes the additive ambiguity of $x_1=-\log |S_1|$.

The main existence theorem and asymptotic information of the Tian-Yau metric is summarized as follows:

\begin{thm}\label{TianYauthm}
	\cite{TianYau}\cite[Prop. 2.9]{Heingravitational} There is a complete Ricci-flat K\"ahler metric 
	\[
	\omega_{TY}= dd^c \phi_{TY}=\omega_{Cal'}+ dd^c \phi_{TY,rel}
	\] solving the complex Monge-Amp\`ere equation
	\[
	\omega_{TY}^{n-1} = \frac{d_2^{n-2}\int_Y c_1(L_0)^{n-2}}{4\pi} \sqrt{-1}^{(n-1)^2} \Omega_{D_1}\wedge \overline{\Omega}_{D_1},
	\]
	with exponential decay estimate for some constant $c>0$ depending on $D_1, D_2$:
	\[
	|\nabla_{\omega_{Cal'}}^k \phi_{TY,rel} |_{ \omega_{Cal'} } = O( e^{-c r_{D_1}^{ \frac{n-1}{n} } } ), \quad r_{D_1}\to +\infty, \quad k\geq 0.
	\]
	
\end{thm}

The volume growth rate of geodesic balls on the Tian-Yau space is $\text{Vol}(B(r_{D_1}))\sim r_{D_1}^{\frac{2(n-1)}{n} }$, which is \emph{less than quadratic}. As such, solving the \textbf{Poisson equation} with potential decaying at infinity would require an integral normalization on the forcing term, which is related to (\ref{integralnormalization}) as
the Poisson equation is the linearization of the complex Monge-Amp\`ere equation.

\begin{prop}
Let $f$ be a smooth function on the Tian-Yau space satisfying the integral normalization $\int f\omega_{TY}^{n-1}=0$ and the fast decay
\[
|\nabla_{TY}^k f|=O(  e^{-c r_{D_1}^{ \frac{n-1}{n} } }   ),\quad \forall k\geq 0,
\]
then there is a smooth function $u$ with $\Lap_{TY} U=f$ with fast decay
\[
|\nabla_{TY}^k U| =O(  e^{-c r_{D_1}^{ \frac{n-1}{n} } }   ), \quad \forall k\geq 0,
\]
where the decay rate $c>0$ may be shrinked.
\end{prop}

\begin{rmk}
The existence of solution with $L^\infty$ bound and $L^2$-gradient bound follows from \cite[Thm 1.5]{HeinSobolev}. The exponential type decay estimate is because the Tian-Yau space is $\text{CYL}(\frac{1}{n} )$ in the sense of Hein \cite[Def. 2.5]{Heingravitational}, and the proof of \cite[Prop. 2.9]{Heingravitational} works almost verbatim for the Poisson equation.
\end{rmk}

To solve the Poisson equation without the integral normalization condition on the forcing term, we need to allow solutions growing at infinity. Consider the special function $(-\log |S|)^{ \frac{1}{n-1}  }$ defined near infinity, which we extend smoothly to a global function $u_0$ on the Tian-Yau space $D_1\setminus D_2$.

\begin{lem}
The function $f_0= \Lap u_0$ satisfies the fast decay near infinity
\[
|\nabla_{TY}^k f_0 |= O(  e^{-c r_{D_1}^{\frac{n-1}{n } } } )
\]
and the total integral
\[
\frac{1}{2\pi(n-1)}\int_{D_1\setminus D_2} f_0\omega_{TY}^{n-1}=
\int_{D_1\setminus D_2} dd^c u_0 \wedge \omega_{TY}^{n-2}= \frac{d_2^{n-2}}{n-1}\int_Y c_1(L_0)^{n-2} \neq 0.
\]
\end{lem}

\begin{proof}
The exponential type decay is because the deviation between the Tian-Yau metric and the Calabi ansatz is exponentially small, and on the Calabi ansatz model $(-\log |S|)^{ \frac{1}{n-1} }$ is precisely harmonic. This is the intimately connected to the freedom to add a constant to $-\log |S|$ in the Calabi ansatz, without affecting the complex Monge-Amp\`ere measure.

To evaluate $\int f_0\omega_{TY}^{n-1}$, we take a very large compact set $K= \{   \log |S|\leq R\}$, and consider the $R\to \infty$ limit. We have up to exponentially suppressed errors
\[
\int_K dd^c u_0 \wedge \omega_{TY}^{n-2}= \int_{\partial K} d^c u_0 \wedge \omega_{TY}^{n-2} \approx \int_{\partial K} d^c u_0 \wedge \omega_{Cal'}^{n-2}.
\]
Computing in the Calabi ansatz model, this is 
\[
-\frac{1}{n-1} \int_{\partial K} d^c \log |S|\wedge (dd^c (-\log |S|))^{n-2}
\]
which is
\[
\frac{1}{n-1} \int_Y  (dd^c (-\log |S|))^{n-2}= \frac{1}{n-1}\int_Y (d_2c_1(L_0))^{n-2}= \frac{d_2^{n-2}}{n-1}\int_Y c_1(L_0)^{n-2}.
\]
Taking the $R\to +\infty$ gives the total integral.
\end{proof}

\begin{cor}
Let $f$ be a smooth function on the Tian-Yau space satisfying the fast decay
\[
|\nabla_{TY}^k f|=O(  e^{-c r_{D_1}^{ \frac{n-1}{n} } }   ),\quad \forall k\geq 0,
\]
then there is a smooth function $U$ with $\Lap_{TY} U=f$ in the form 
\[
U=\text{const}\cdot u_0+\tilde{U},\quad 
|\nabla_{TY}^k U| =O(  e^{-c r_{D_1}^{ \frac{n-1}{n} } }   ), \quad \forall k\geq 0,
\]
for some possibly shrinked $c>0$.
\end{cor}

\subsection{Non-generic region near infinity}\label{Nongenericregiongluingansatz}

Next, we need to move up one dimension, and produce some ansatz metric on $X=\bar{X}\setminus D_1\cup D_2$ in the region near infinity close to $D_1\setminus D_2$, corresponding in the logarithmic coordinate to $x_1\gg |x_2|+1$, including in particular the region with $x_2=O(1), x_1\gg 1$. Pick an auxiliary smooth Hermitian metric on $\bar{X}$, such that the normal vector field to $D_1\subset \bar{X}$ is tangent to $D_2$ along $Y=D_1\cap D_2$. Then normal geodesic flow identifies a tubular neighbourhood of $D_1$ with the normal bundle of $D_1$, and the error caused by the failure of holomorphicity is exponentially suppressed in the log coordinates. We can thus regard the potential $\phi_{TY}$ of the Tian-Yau metric on $D_1\setminus D_2$ as a function on its tubular neighbourhood, via pullback.

To match with the asymptote (\ref{upowerseries}) in the region $1\ll x_2\ll x_1$,
 we are motivated to consider the  \textbf{local ansatz potential} near $D_1\setminus D_2$
\begin{equation}
\phi_{D_1}= v_0 \tilde{x}_1^{\frac{n+2}{n} }   +   a (\frac{d_1}{d_2})^{\frac{n}{n-1}}  \tilde{x}_1^{ \frac{n-2}{n(n-1)}  } \phi_{TY}+\sum_{k\geq 2} a_k \tilde{x}_1^{ \frac{n+2}{n} -\frac{kn}{n-1} } (\frac{ d_1x_2}{d_2})^{\frac{kn}{n-1}  } .
\end{equation}
Here we make a fixed choice of a positive valued smooth function on the Tian-Yau space matching with $x_2$ outside a compact set, still denoted as $x_2$, so that when $x_2=O(1)$ the expression $\phi_{D_1}$ remains smooth. The convergence of the series is valid for $x_2\ll x_1$. Of course, the ad hoc choice means that the ansatz needs to be corrected later by more refined terms.

The local geometry should be imagined as a fibration of Tian-Yau metrics with slowly changing size depending on the logarithmic variable $x_1$. Consider $\tilde{x}_1$ around a given large value $x_1'\gg 1$. The dominant terms in $dd^c \phi_{D_1}$ are 
\begin{equation}\label{productlocalmetric}
\omega_{x_1'}=a (\frac{d_1}{d_2})^{n/(n-1)} x_1'^{ \frac{n-2}{n(n-1)}  }  dd^c \phi_{TY}+ \frac{2(n+2)v_0}{n^2d_2^2} x_1'^{-\frac{n-2}{n} } \frac{\sqrt{-1}}{4\pi} d\log \xi'\wedge d\overline{\log \xi'}.
\end{equation}
Up to taking a finite cover (to do with taking fractional powers of $\xi'$), this model metric describes the product of the Tian-Yau metric with length scale $O(\rho)  $ corresponding to the harmonic radius scale, and a cylinder of circle length scale $O(x_1'^{\frac{2-n}{2n} })$ corresponding to the injectivity scale.

We can now set up the \textbf{local H\"older norms} in the nongeneric region $x_2\ll x_1$ (The case of $x_1\ll x_2$ is completely similar).
When $x_2=O(1), x_1\gg 1$, we unwind the $S^1$ factor of the cylinder to pass to the local universal cover.    On $O(\rho)$ neighbourhoods,  we can use the local product metric (\ref{productlocalmetric}) to define the local H\"older seminorm and the $C^{k,\alpha}$-norms, as in (\ref{localHolderseminorm})(\ref{localHolder}).

The case of $1\ll x_2\ll x_1$ is already covered by section \ref{Genericregiongluingansatz}. If we had used the local product metric (\ref{productlocalmetric}) as reference metric, it would lead to an equivalent definition of local H\"older norms up to uniform equivalence.  Notice that within $O(\rho)$ neighbourhoods the values of $x_1, x_2$ do not vary drastically, so expressions like $O(x_1^\alpha x_2^\beta)$ are not sensitive to the choice of points in the $O(\rho)$ neighbourhood, except when $x_2=O(1)$, in which case we use a slightly abusive convention that $O(x_2^\beta)$ stands for $O(1)$ in this region.

We start by observing the derivative bounds of basic functions, which can be read off using the leading order metric ansatz, remembering that the Tian-Yau metric is well approximated by the Calabi ansatz except when $x_2=O(1)$, and that the $x_1'$-dependence in $\omega_{x_1'}$ introduces scaling factors.

\begin{lem}\label{basicderivative}
	In the region $\{  x_2\leq x_1, \text{and } x_1\gg 1  \}$,
	\[
	\norm{dx_1}_{k,\alpha,loc}= O ( x_1^{ \frac{n-2}{2n}  } ),\quad   \norm{dx_2}_{k,\alpha,loc}= O( x_1^{ -\frac{n-2}{2n(n-1)}  } x_2^{ \frac{n-2}{2(n-1)}  }  ),
	\]
	\[
	dd^c (x_1-\frac{d_1}{d_2}x_2 ) =0,\quad \norm{dd^c x_2}_{k,\alpha,loc}=O( x_2^{ -\frac{1}{n-1} } x_1^{ -\frac{n-2}{n(n-1)} }  ).
	\]
\end{lem}

\begin{lem}
For $x_2\ll x_1$, the Tian-Yau potential satisfies the bound
\[
|\phi_{TY}|=O( x_2^{ \frac{n}{(n-1)}     } ), \quad  |d\phi_{TY}|=O( x_1^{ -\frac{n-2}{2n(n-1)}  } x_2^{ \frac{n}{2(n-1)} } ),  
\]
\[
\norm{dd^c\phi_{TY}}_{k,\alpha,loc}= O( x_1^{ - \frac{n-2}{n(n-1)} } ).
\]
\end{lem}

The following lemma quantifies the approximation of the local ansatz potential $\phi_{D_1}$ by the model product metric:

\begin{lem}\label{Metricdeviation1}
(Metric deviation) 
For $x_1'\gg 1$, then in the region with
$|x_1-x_1'|\leq x_1'^{\frac{n-2}{2(n-1)}  }$ and $|x_2|\ll x_1'$, the metric deviation between $dd^c \phi_{D_1}$ and the local product metric (\ref{productlocalmetric}) satisfies the local $C^{k,\alpha}$ estimate on $O(\rho)$ neighbourhoods around a given point:
\[
\norm{ dd^c \phi_{D_1} - \omega_{x_1'} }_{k,\alpha,loc}= O( (\frac{x_2}{x_1} )^{ \frac{n}{2(n-1) }  }   )
\]

\end{lem}

\begin{proof}
We will ignore all exponentially small errors coming from identifying tubular neighbourhoods with normal bundles. We compute $dd^c\phi_{D_1}$ by the Leibniz rule:
\[
dd^c \tilde{x}_1^{ \frac{n+2}{n}  }= \frac{2(n+2)}{n^2}\tilde{x}_1^{ \frac{n+2}{n}-2  } \frac{ \sqrt{-1}}{ 4\pi d_2^2 }d\log \xi'\wedge d\overline{\log \xi'}, 
\]
\[
\begin{split}
& dd^c (\tilde{x}_1^{ \frac{n-2}{n(n-1)}  }\phi_{TY}) 
=\frac{n-2}{n(n-1)} (  \frac{n-2}{n(n-1)}-1) \tilde{x}_1^{ \frac{n-2}{n(n-1)}-2 } \phi_{TY} \frac{ \sqrt{-1}}{ 4\pi d_2^2 }d\log \xi'\wedge d\overline{\log \xi'} 
\\
&+ \frac{n-2}{n(n-1)}\tilde{x}_1^{ \frac{n-2}{n(n-1)}-1 } ( d\tilde{x}_1\wedge d^c\phi_{TY} + d\phi_{TY}\wedge d^c \tilde{x}_1) + \tilde{x}_1^{ \frac{n-2}{n(n-1)}  }dd^c\phi_{TY}.
\end{split}
\]
and similarly with $dd^c (\tilde{x}_1^{  \frac{n+2}{n}- \frac{kn}{n-1} }x_2^{ \frac{kn}{n-1}  }) $.

Comparing $v_0dd^c\tilde{x}_1 ^{\frac{n+2}{n} } $ and the term $\frac{2(n+2)v_0}{n^2d_2^2} x_1'^{-\frac{n-2}{n} } \frac{\sqrt{-1}}{4\pi} d\log \xi'\wedge d\overline{\log \xi'}$, the deviation is of order $O( \frac{|x_2|+ |x_1-x_1'|}{x_1'}   )$. Notice for $n\geq 3$, we have $\frac{n}{2(n-1)}<1$ and so $\frac{|x_2|}{x_{1}'} \leq \left(\frac{x_2}{x_1}\right)^{\frac{n}{2(n-1)}}$. For $|x_1-x_1'|\leq x_1'^{\frac{n-2}{2(n-1)}  }$, the error $O( \frac{|x_1-x_1'|}{x_1'}  )=O( x_1^{-\frac{n}{2(n-1)}  }  )$, which is absorbed into $O( (\frac{x_2}{x_1})^{\frac{n}{2(n-1)}   }  )$.

Using the lemmas, we can estimate the terms appearing in $dd^c (\tilde{x}_1^{ \frac{n-2}{n(n-1)}  }\phi_{TY}) $ in the local $C^{k,\alpha}$-norms:
\[
\norm{ \phi_{TY}x_1^{ \frac{n-2}{n(n-1)}-2   } d\log \xi'\wedge d\overline{\log \xi'} }_{k,\alpha,loc}= O(x_2^{ \frac{n}{(n-1)}  }x_1^{ \frac{n-2}{n(n-1)}-2   } x_1^{\frac{n-2}{n}  } )= O((\frac{x_2}{x_1}  )^{ \frac{n}{n-1}  }).
\]
The cross terms have order
\[
\norm{
x_1^{ \frac{n-2}{n(n-1)}-1  } d\log \xi'\wedge d^c\phi_{TY} }_{k,\alpha,loc}=O( x_1^{  \frac{n-2}{n(n-1)}-1  }  x_1^{\frac{n-2}{2n}  }   x_1^{ -\frac{n-2}{2n(n-1)}  } x_2^{ \frac{n}{2(n-1)} } ) = O( ( \frac{x_2}{x_1} )^{ \frac{n}{2(n-1) } } ).
\]
So does its complex conjugate. The last error comes from the deviation between $\tilde{x}_1^{ \frac{n-2}{n(n-1)}  }dd^c \phi_{TY}$ from  $x_1'^{ \frac{n-2}{n(n-1)} }dd^c\phi_{TY}$. This error is again of order $O( \frac{|x_2|+ |x_1-x_1'|}{x_1'}   )$.

The remainder terms have leading order contribution
$
dd^c (  \tilde{x}_1^{ \frac{n+2}{n}- \frac{2n}{n-1}  }  x_2^{\frac{2n}{n-1}  }   ).
$
Its main contribution is $\tilde{x}_1^{ \frac{n+2}{n}- \frac{2n}{n-1}  } dd^c  x_2^{\frac{2n}{n-1}  } $, which has magnitude $O( (\frac{x_2}{x_1})^{ \frac{n}{n-1}  }  )$.
Combining all the errors, the deviation between $dd^c \phi_{D_1}$ and the local product metric (\ref{productlocalmetric}) has magnitude bounded by
$
 O( (\frac{x_2}{x_1} )^{ \frac{n}{2(n-1) }  }  ).
$
\end{proof}

\begin{lem}\label{VolumeerrorphiD1}
(Volume form error) In the region with $x_1\gg 1$ and $x_2\ll x_1$,
\[
\norm{(dd^c\phi_{D_1})^n -K_0\sqrt{-1}^{n^2}\Omega\wedge \overline{\Omega} }_{k,\alpha,loc} =O( x_1^{- \frac{n}{(n-1)} } e^{-c x_2^{1/2}}   ).
\]
\end{lem}

\begin{proof}
The volume form error will be smaller than the metric deviation due to extra cancellation effects.
Recall the computation of $dd^c\phi_{D_1}$ from Lemma \ref{Metricdeviation1} above. 
Notice $(dd^c\phi_{D_1})^n$ is a volume form, so must take one $d\log \xi'$ and $d\overline{\log \xi'}$ from either a pair of cross derivative terms, or from a factor $dd^c \tilde{x}_1^{ \frac{n+2}{n}-  \frac{kn}{n-1} }$ for $k\geq 0$. The differentiation of $\phi_{TY}$ and the powers of $x_2$ would only produce factors on $D_1\setminus D_2$ without $\tilde{x}_1$ dependence. By thinking about all the possible ways to take wedge products contributing to $(dd^c\phi_{D_1})^n$, we get an absolutely convergent series within $x_2\ll x_1$:
\begin{equation}\label{Ik}
(dd^c\phi_{D_1})^n\approx \sum_{k\geq 0} \tilde{x}_1^{ - \frac{kn}{n-1}  } \sqrt{-1} d\log \xi'\wedge d\overline{\log \xi'} \wedge \mathcal{I}_k,
\end{equation}
where the coefficients $\mathcal{I}_k$ are top degree forms on the $D_1\setminus D_2$ factor, without $\tilde{x}_1$ dependence. Here we write $\approx$ as a reminder that we have ignored the exponentially small errors from identifying the tubular neighbourhood with the normal bundles of $D_1\setminus D_2$, which is holomorphically trivial up to finite cover.

We now identify the leading term $d\log \xi'\wedge d\overline{\log \xi'}\wedge \mathcal{I}_0$. By thinking about form types, this comes from the binomial expansion term
\[
n v_0dd^c\tilde{x}_1 ^{\frac{n+2}{n} } \wedge \left(    a (\frac{d_1}{d_2})^{ \frac{n}{n-1} } \tilde{x}_1^{ \frac{n-2}{n(n-1) } } dd^c \phi_{TY}          \right)^{n-1} ,
\]
which is
\[
\frac{2(n+2)v_0}{nd_2^2}  \frac{\sqrt{-1}}{4\pi} d\log \xi'\wedge d\overline{\log \xi'}\wedge  a^{n-1} (\frac{d_1}{d_2})^{n}   (dd^c\phi_{TY})^{n-1}.
\]
From the explicit formula (\ref{v0a}) for $v_0, a$,
\[
v_0a^{n-1}= (\frac{n}{n+2})^3,
\]
and recalling the complex Monge-Amp\`ere equation for the Tian-Yau metric in Theorem \ref{TianYauthm}, the above simplifies to
\[
\frac{2n^2\int_Y c_1(L_0)^{n-2}}{(n+2)^2 d_2^2} (\frac{d_1}{d_2})^{n} d_2^{n-2}  \frac{\sqrt{-1}}{4\pi} d\log \xi'\wedge d\overline{\log \xi'}\wedge  \sqrt{-1}^{(n-1)^2}  \Omega_{D_1}\wedge \overline{\Omega}_{D_1}.
\]
Up to exponentially small errors from complex structure identifications, in terms of the local defining functions $\xi_1,\xi_2$ for $D_1, D_2$,
\[
\Omega_{D_1}\approx d\log \xi_2\wedge \Omega_Y, \quad \log \xi'\approx d_2\log \xi_1- d_1\log \xi_2,\quad \Omega\approx d\log \xi_1\wedge d\log \xi_2\wedge \Omega_Y,
\]
and the above reduce to
$
K_0 \sqrt{-1}^{n^2}\Omega\wedge \overline{\Omega},
$
for the constant $K_0$ in (\ref{K0}). In short, the leading term cancels with $
K_0 \sqrt{-1}^{n^2}\Omega\wedge \overline{\Omega}
$.

The subleading terms are suppressed by the factor $\tilde{x}_1^{- \frac{n}{n-1} }$, and the fast convergence of the power series means we only need to consider $\mathcal{I}_1$. In the Tian-Yau core region $x_2=O(1)$, the crude information is that
$d\log \xi'\wedge d\overline{\log \xi'}\wedge \mathcal{I}_1$ has local $C^{k,\alpha}$ norm $O(1)$. This explains the $O(x_1^{-\frac{n}{n-1} })$ decay in the $x_2=O(1) $ subregion.

For $1\ll x_2\ll x_1$, the deviation between the Tian-Yau potential $\phi_{TY}$ and the Calabi ansatz potential is $O(e^{-cx_2^{1/2}} )$, namely $O(e^{-cr_{D_1}^{\frac{n-1}{n} } })$ (for some changing constant $c>0$). After replacing $\phi_{TY}$ by $\frac{n-1}{n}x_2^{\frac{n}{n-1} }$, we recover the potential $u$ in the generic region. Ignoring exponentially small complex structure errors as usual, then $u$ is by construction a solution to the complex Monge-Amp\`ere equation. This explains the exponential decay $\mathcal{I}_k=O(e^{-cx_2^{1/2} } )$ for $k\geq 1$.
\end{proof}

\subsection{Refined local ansatz in the non-generic region}

Recall the distance to the origin is $O( |x|^{ \frac{n+2}{2n}  } )$. Thus the $O(x_1^{ -\frac{n}{n-1} })$ decay is \emph{slower than quadratic}, and we need further correction terms to improve the ansatz. The linearization of the complex Monge-Amp\`ere equation will naturally lead to a Poisson equation.

Using section \ref{TianYaubackground}, we can solve (a rescaled version of) the Poisson equation on the Tian-Yau space $D_1\setminus D_2$:
\begin{equation}
\frac{(n+2)(n-1)}{2\pi nd_2^2} v_0a^{n-2} (\frac{d_1}{d_2})^{ \frac{n(n-2)}{n-1}  }(dd^c U)\wedge \omega_{TY}^{n-2} =  -\mathcal{I}_1,
\end{equation}
where we recall from (\ref{Ik}) that $\mathcal{I}_1$ is a top degree form on $D_1\setminus D_2$, with exponential decay $O(e^{-cx_2^{1/2}})$ for some $c>0$ to all orders of derivatives. The caveat is that $\int_{D_1\setminus D_2} \mathcal{I}_1$ is not guaranteed to be zero, so $U$ may not decay at infinity. Instead,
\[
U= a' u_0+ \tilde{U}, 
\]
where $|\nabla^k_{TY} \tilde{U}|= O( e^{-cx_2^{1/2}} )$ for some possibly shrinked $c>0$, and $a'$ is a constant.

We regard $U$ as a function in the region $x_2\ll x_1$, namely the tubular neighbourhood around $D_1\setminus D_2$, and 
let $\phi_{D_1}^{(2)}=\phi_{D_1}+ \tilde{x}_1^{ \frac{n+2}{n}- \frac{2n}{n-1}   } U$. From the leading term $u_0\sim x_2^{\frac{1}{n-1}}$ in $U$, we can compute using Lemma \ref{basicderivative} that the correction term $\tilde{x}_1^{ \frac{n+2}{n}- \frac{2n}{n-1}   } U$ has local $C^{k,\alpha}$-norm $O( x_1^{-\frac{n}{n-1} }x_2^{-1}  )$. This small correction term leads to better volume form decay:

\begin{lem}\label{VolumeformerrorphiD12}
In the region $|x_2|+1\ll x_1$, we have the improved decay
\[
\norm{ (dd^c \phi_{D_1}^{(2)})^n -K_0\sqrt{-1}^{n^2}\Omega\wedge \overline{\Omega} }_{k,\alpha,loc} = O( x_1^{-\frac{2n}{n-1} } x_2^{\frac{1}{n-1}} ).
\] 
\end{lem}

\begin{proof}
We revisit the calculations in Lemma \ref{VolumeerrorphiD1}. The volume form $(dd^c\phi_{D_1}^{(2)})^n$ should be viewed as a perturbation of $(dd^c\phi_{D_1})^n$.
Again $(dd^c\phi_{D_1}^{(2)})^n$ is a convergent power series of $\tilde{x}_1^{-\frac{n}{n-1} }$, with coefficient in top degree forms on $D_1\setminus D_2$. The leading order contribution to $(dd^c\phi_{D_1}^{(2)})^n- (dd^c\phi_{D_1})^n$ is
\[
n(n-1)v_0dd^c \tilde{x}_1^{ \frac{n+2}{n}  }\wedge (  a (\frac{d_1}{d_2})^{\frac{n}{n-1}}  \tilde{x}_1^{ \frac{n-2}{n(n-1)}  } dd^c\phi_{TY}     )^{n-2}\wedge \tilde{x}_1^{ \frac{n+2}{n}- \frac{2n}{n-1}   }  dd^c U,
\]
which after some calculation gives
\[
\frac{2(n+2)(n-1)}{ nd_2^2} v_0a^{n-2} (\frac{d_1}{d_2})^{ \frac{n(n-2)}{n-1}  }(dd^c U)\wedge \omega_{TY}^{n-2}\wedge \tilde{x}_1^{ -\frac{n}{n-1}   } \frac{\sqrt{-1}}{4\pi}d\log \xi'\wedge d\overline{\log \xi'}.
\]
This is by construction 
$
-  \sqrt{-1}d\log \xi'\wedge d\overline{\log \xi'}\wedge \mathcal{I}_1,
$
which is precisely designed to cancel the leading order error $ \sqrt{-1}d\log \xi'\wedge d\overline{\log \xi'}\wedge \mathcal{I}_1$ in (\ref{Ik}).

The next order of error has $\tilde{x}_1^{ -\frac{2n}{n-1} }$ in front. Within the $x_2=O(1)$ region, the volume error is now $O( \tilde{x}_1^{ -\frac{2n}{n-1} } )$. For $x_2\gg 1$ one needs to be careful about the effect of $U$ having a growing term $a'u_0$, which will damage the exponential decay. Recall from section \ref{TianYaubackground} that $u_0=x_2^{\frac{1}{n-1}  }$ outside some compact region in the Tian-Yau space. The largest new contributions to the volume forms error come from terms such as
\[
(   \tilde{x}_1^{ \frac{n-2}{n(n-1)}  } dd^c\phi_{TY}   )^{n-2} \wedge d^c \tilde{x}_1^{ \frac{n-2}{n(n-1)}  }\wedge d\phi_{TY}    \wedge d\tilde{x}_1^{ \frac{n+2}{n}- \frac{2n}{n-1}   } \wedge d^c u_0  ,
\]
whose local $C^{k,\alpha}$-norm is $O(\tilde{x}_1^{ -\frac{2n}{n-1} } x_2^{ \frac{1}{n-1}  }  )$. 
\end{proof}

\subsection{Gluing the regions}

We now glue the potential $u$ in the generic region $1\ll \min\{x_1, x_2 \}$, and the potential $\phi_{D_1}^{(2)}$ in the region $x_1\gg |x_2|+1 $ (and a completely similar potential $\phi_{D_2}^{(2)}$ in the region $x_2\gg |x_1|+1$). We now take a smooth cutoff function $\eta:\R\to [0,1]$, with
\begin{equation*}\label{cutoff}
\eta(x)=
\begin{cases}
1, \quad x\leq 1,\\
0,\quad x\geq 2. 
\end{cases}
\end{equation*}
The glued potential is defined for $|x_1|+|x_2|\gg 1$,
 \begin{equation}
 \phi_{glue}= \eta( \tilde{C}\frac{ x_2}{x_1}  ) (  \phi_{D_1}^{(2)} - u ) + \eta( \tilde{C} \frac{x_1}{x_2}  ) (  \phi_{D_2}^{(2)} - u ) +u,
 \end{equation}
where $\tilde{C}\gg 1$ is some large fixed constant, specifying the gluing region $x_1\sim \tilde{C} x_2$ (resp. $x_2\sim \tilde{C}x_1$). For large $x_1\geq \tilde{C}\tilde x_2$, then $\phi_{glue}$ agrees with $\phi_{D_1}^{(2)}$, while for $2\tilde{C}\tilde x_2\leq x_1\leq  (2\tilde{C})^{-1}x_2$, then $\phi_{glue}$ agrees with $u$.

\begin{lem}\label{Metricdeviationerror2}
For $|x_1|+|x_2|\gg 1$, the local $C^{k,\alpha}$-norm of the metric gluing error 
\[
\norm{ \eta( \tilde{C}\frac{ x_2}{x_1}  ) (\phi_{D_1}^{(2)}-u) }_{k,\alpha,loc} = O( x_1^{ - \frac{2n-1}{n-1} } ).
\]
In particular the glued metric remains K\"ahler.
\end{lem}

\begin{proof}
In the gluing region $x_1\sim \tilde{C}x_2$, namely $x_1, x_2$ and $\tilde{x}_1$ are comparably large, the deviation between $\phi_{D_1}^{(2)}$ and $u$ comes from the $O(e^{-cx_2^{1/2}} )$ small deviation between the Tian-Yau potential and the Calabi ansatz, and the correction term $\tilde{x}_1^{\frac{n+2}{n}- \frac{2n}{n-1}   }U$. Ignoring the exponentially small effects, the only important term is $\tilde{x}_1^{\frac{n+2}{n}- \frac{2n}{n-1}   }x_2^{\frac{1}{n-1}}$. We compute from Lemma \ref{basicderivative}
\[
\norm{ dd^c ( \eta \tilde{x}_1^{\frac{n+2}{n}- \frac{2n}{n-1}   }x_2^{\frac{1}{n-1}} )}_{k,\alpha, loc} = O( x_1^{\frac{n+2}{n}- \frac{2n}{n-1} +\frac{1}{n-1}-2+ \frac{n-2}{n}  }  )  = O(x_1^{ - \frac{2n-1}{n-1} }   ).
\]
\end{proof}

We also need to extend the gluing ansatz to a  K\"ahler metric over the compact region with $x_1, x_2=O(1)$.  We can first extend $\phi_{glue}$ to a smooth potential over $X$, which may not be K\"ahler inside a fixed compact set. Taking a very ample linear system  $H^0(\bar{X}, m(d_1+d_2)L_0)$ for $m\gg 1$, we can construct a Fubini-Study metric on $\bar{X}$. Using the defining section $S_1S_2\in H^0(\bar{X}, (d_1+d_2)L_0)$ of the divisor $D_1+D_2$, we can regard the Fubini-Study metric as a function on $X=\bar{X}\setminus D_1\cup D_2$, of the form
\[
\phi_{FS}=\log \sum_i  |\frac{ s_i}{(S_1S_2)^m} |^2.
\]
Clearly $\phi_{FS}$ tends to infinity near $D_1\cup D_2$ at a speed comparable to $x_1+x_2$. 
We take some large constant $A\gg 1$, and $R\gg 1$ depending on $A$, and add to $\phi_{glue}$ the term
$
A \phi_{FS} \eta(  \frac{ \phi_{FS} }{R}       ).
$
Intuitively, the cutoff function $\eta$ turns off the Fubini-Study potential outside a large compact subset. By making $A$ large enough, we can improve $\phi_{glue}$ to be K\"ahler in a fixed compact set. For $R\gg 1$, the cutoff error in the region $\phi_{FS}\sim R$ is suppressed by $dd^c \phi_{glue}$, and the metric remains positive.
By a slight abuse, we shall continue to use $\phi_{glue}$ to refer to the global K\"ahler metric on $X$.

Define the volume error function $Err_2$ by
\begin{equation}\label{Err2}
(dd^c \phi_{glue})^n = K_0(1+Err_2)\sqrt{-1}^{n^2} \Omega\wedge \overline{\Omega}.
\end{equation}

\begin{cor}\label{Volumeformerrorphiglue}
The volume form error of the glued ansatz is
\[
\norm{Err_2}_{k,\alpha,loc} =O((1+|x_1|+|x_2|)^{- \frac{2n-1}{n-1}}). 
\]

\end{cor}

\begin{proof}
For $|x_2|+1\ll x_1$, Lemma \ref{VolumeformerrorphiD12} says that
the volume form error of $\phi_{D_1}^{(2)}$ is $O( x_1^{- \frac{2n}{n-1}  } x_2^{ \frac{1}{n-1} } )=O(x_1^{- \frac{2n-1}{n-1}  }   )$.	In the gluing region $x_1\sim \tilde{C}x_2$, the volume form error is controlled by the metric gluing error, which is again $O(x_1^{- \frac{2n-1}{n-1}  }   )$ by Lemma \ref{Metricdeviationerror2}. What happens near $D_2$ is completely analogous. Finally, in the compact region $x_1, x_2=O(1)$, the smoothness of $\phi_{glue}$ means that the local $C^{k,\alpha}$-norm is $O(1)$.
\end{proof}

\subsection{Distance-like function}\label{Distancelike}

For our later invocation of Hein's package (\cf section \ref{Hein}), we need to know the existence of distance-like functions with gradient and complex Hessian control.

Define a smooth function 
\[
\tilde{\rho}= (x_1^2+ x_2^2+1)^{  \frac{n+2}{4n}  }.
\]
As discussed in section \ref{Basiclengthscales}, the distance function to the origin is uniformly equivalent outside a compact set to
$
|x|^{ \frac{n+2}{2n}  }, 
$
and $\tilde{\rho}$ can be viewed as a regularized version. An easy consequence of Lemma \ref{basicderivative} is

\begin{lem}
The function $\tilde{\rho}$ satisfies
$
|d\tilde{\rho}|\leq C$ and $ \tilde{\rho}|dd^c\tilde{\rho}|\leq C.
$
\end{lem}

In terms of the distance-like function $\tilde{\rho}$, we can rewrite Cor. \ref{Volumeformerrorphiglue} as 
\begin{equation}\label{volumedecayfasterthanquad}
\norm{Err_2}_{k,\alpha,loc} =O(\tilde{\rho}^{ - \frac{ 2n(2n-1) }{(n-1)(n+2)}   }). 
\end{equation}
Crucially for our later purpose, this decay is \emph{faster than quadratic} $O(\tilde{\rho}^{-2})$ to all orders of derivatives. Morever, the formula for the Ricci form
\[
Ric= -\sqrt{-1}\partial \bar{\partial} \log \left(\frac{  (dd^c \phi_{glue})^n  }{   K_0\sqrt{-1}^{n^2} \Omega\wedge \overline{\Omega}   }   \right)
\]
implies $|Ric|=O( \tilde{\rho}^{ - 2 }  )$ for the glued ansatz.

\section{Weighted Sobolev and Calabi-Yau metric}\label{WeightedSobolevCY}

\subsection{Hein's package}\label{Hein}

Hein \cite[Chapter 3, 4]{Heinthesis} sets out a framework for solving the complex Monge-Amp\`ere equation on complete noncompact manifolds, building on the seminal paper by Tian and Yau \cite{TianYau}. We explain Hein's results in a variant form which follows from his arguments.
The ambient complete manifold $M$ needs to satisfy the following analytic properties:

\begin{itemize}
	\item There is a $C^{k,\alpha}$ quasi-atlas with $k\geq 3$, meaning a collection of charts on which the metric is uniformly equivalent to the Euclidean metric, and the complex structure and the metric has $C^{k,\alpha}$ bounds. Beware that in our applications the injectivity radius can degenerate, and the charts involve local universal covers.

	\item 
	There is a function $\tilde{\rho}$  uniformly equivalent to  $\text{dist}(0,x)+1$, and satisfies $|\nabla \tilde{\rho}|+ \tilde{\rho} |dd^c \tilde{\rho}| \leq C$. This assumption is useful in integration by part arguments.
	
	\item 
	We need the \textbf{weighted Sobolev inequality} on functions: assume the power law volume growth $\text{Vol}(B(r))\sim r^{p'}  $ with rate $p'>2$. 
	For $1\leq p\leq \frac{ \dim_\R M }{\dim_\R M-2}$ and functions $u$ with $L^2$-gradient,
	\[
	(\int |u|^{2p} \tilde{\rho}^{ p(p'-2)-p'  }dvol )^{1/p} \leq C\int |\nabla u|^2.
	\]
	These inequalities differ from the standard Sobolev inequalities in the sense that they do not require the manifold to have Euclidean volume growth, which makes them remarkably flexible.

\end{itemize}

The output of this package is:
\begin{itemize}
	
	\item  Denote $\omega_0$ as the ambient K\"ahler form. Let $f\in C^{2, \alpha}$ satisfy $|f|\leq C \tilde{\rho}^{ -q  }$ for $p'>q>2$. Then there is some $0< \alpha'\leq \alpha$ and $u\in C^{4, \alpha'}$ which solves $(\omega_0+ dd^c u )^{\dim_\C M }=e^f \omega_0^{ \dim_\C M  }$,
	with decay estimate $|u|\leq C \tilde{\rho}^{2-q+\epsilon}$, where $\epsilon\ll 1$ is any fixed small number. 
\end{itemize}

Here we have separated the assumptions on the ambient manifolds from the decay assumptions to emphasize that these are difficulties of distinct nature. The key idea in Hein's package is to obtain a priori $L^\infty$ estimates and power law decay estimates on potentials via the method of weighted Moser iteration, which hinges on the weighted Sobolev inequalities. The estimates from Hein's package are constructive. It is essential to assume \textbf{faster than quadratic decay} on the source function $f$, because the method needs the potential $u= O(\rho^{2-q})$ to be bounded.

\subsection{Sufficient condition for weighted Sobolev}\label{WeightedSobolevcondition}

Now recall \cite[Def. 1.1]{HeinSobolev} a complete manifold $(M,g)$ is called $\text{SOB}(p')$, if there exist $x_0\in M$ and $C\geq 1$ such that 
\begin{itemize}
\item $B(x_0, s)\setminus B(x_0, t)$ is connected for all $s\geq t\geq C$,
\item $\text{Vol}(B(x_0, s))\leq Cs^{p'}$ for all $s\geq C$,
\item $\text{Vol}(B(x, (1-C^{-1})r(x))) \geq C^{-1}r(x)^{p'}$ and $Ric(x)\geq -Cr(x)^{-2}$ if $r(x)=dist(x_0, x)\geq C$. 
\end{itemize}
Hein \cite{HeinSobolev} shows that when $p'>2$, then $SOB(p')$ is a sufficient conditions for the weighted Sobolev inequality. As observed in \cite[section 7]{Gabor}, the connectivity of the annulus can be relaxed to the weaker requirement of \emph{relative connected annulus}: 

\begin{itemize}
\item For sufficiently large $D$, any two points $x_1, x_2 \in M $ with $d(x_0, x_i) =
D$ can be joined by a curve of length at most $CD$, lying in the annulus $B(x_0, CD) \setminus
B(x_0, C^{-1}D)$, for a uniform constant $C>1$.
\end{itemize}

For our particular metric ansatz $dd^c\phi_{glue}$, the assumptions on the volume growth rate of balls are immediate consequences of the our much more refined description of the generalized Calabi ansatz, and the metric behaviour near the Tian-Yau region. The quadratic decay on the Ricci tensor is a consequence of the faster than quadratic decay on the volume form error to all derivatives (\cf section \ref{Distancelike}).

To verify the relative connnected annulus property, notice any point close to the Tian-Yau region can be first connected via a path of length $O(\tilde{\rho})$ contained inside the annulus, to a point in the generic region where $x_1, x_2$ are comparable, and the statement is obvious for two points in the generic region.

\subsection{Assembling the pieces}

\begin{thm}
(\textbf{Existence of complete Calabi-Yau metric})
There is a potential $\phi_{rel}$ such that $\phi=\phi_{glue}+\phi_{rel}$ solves the complex Monge-Amp\`ere equation with decay bound
\[
( dd^c\phi)^n =K_0 \sqrt{-1}^{n^2}\Omega\wedge \overline{\Omega},\quad \norm{\phi_{rel} }_{k,\alpha, loc} =O( \tilde{\rho}^{ -q } ),
\]
where $q$ can be chosen as any positive number smaller than $\frac{2n-4}{n+2}  $.
\end{thm}

\begin{proof}
We have verified the weighted Sobolev inequality (\cf section \ref{WeightedSobolevcondition}), the existence of the distance-like function (\cf section \ref{Distancelike}), and the existence of $C^{k,\alpha}$ quasi-atlas (\cf section \ref{Genericregiongluingansatz}, \ref{AsymptoticexpansionnearD1}).

The volume form error for $dd^c\phi_{glue}$ decays like
$O(\tilde{\rho}^{ - \frac{2n(2n-1)}{(n-1)(n+2)} } )$ (\cf (\ref{Err2})). On the other hand, the volume growth rate is $O(\tilde{\rho}^{p'})$ with $2<p'= \frac{4n}{n+2}<\frac{2n(2n-1)}{(n-1)(n+2)} $. 
In particular the volume error is bounded by $O( \tilde{\rho}^{-p'} )$. Applying Hein's package, we can find a $C^{k,\alpha}$ bounded solution $\phi_{rel}$, such that
\[
(dd^c\phi_{glue}+ dd^c \phi_{rel})^n =  (1+Err_2)^{-1} (dd^c\phi_{glue})^n= K_0 \sqrt{-1}^{n^2}\Omega\wedge \overline{\Omega},
\]
with decay estimate $|\phi_{rel}|=O(\tilde{\rho}^{2-p_0+\epsilon}  )=O(\tilde{\rho}^{-q})$. Since the charts on the local universal covers have harmonic radius scale $O(\rho)$, elliptic regularity improves the decay estimate to local $C^{k,\alpha}$-norms.
\end{proof}

The smallness of $dd^c\phi_{rel}$ near infinity means that the main features of the asymptotic geometry are preserved. In particular, the \textbf{volume growth} of the Calabi-Yau metric $dd^c\phi$ is $\text{Vol}(B(\tilde{\rho}))\sim \tilde{\rho}^{ \frac{4n}{n+2}}$. The \textbf{tangent cone at infinity} refers to the pointed Gromov-Hausdorff limit of rescaled geodesic balls centred at a fixed reference point. In our case, the rescaling procedure obliviates the $T^2$ and $Y$-fibres . The tangent cone is topologically 
$\R_{\geq 0}^2$ with the variables $x_1, x_2$, and metrically it is up to a constant the \textbf{Hessian metric} 
\[
g_\infty= \frac{\partial^2 u}{\partial x_i\partial x_j }dx_idx_j,
\]
where $u(x_1, x_2)$ is the solution to the \textbf{non-archimedean Monge-Amp\`ere equation}. The \textbf{renormalized measure} on the tangent cone, which comes from pushing forward the complex Monge-Amp\`ere measure, is up to constant factor  $dx_1dx_2$.

\section*{Appendix: Hypergeometric functions}

Recall that the hypergeometric function $\HGF[a,b;c;z]$ is defined by
\begin{equation}\label{Hypergeometric}
\HGF[a,b;c;z]= \sum_{n=0}^{\infty} \frac{(a)_n(b)_n}{n!(c)_n} z^n
\end{equation}
where 
\[
(\alpha)_n:= \alpha(\alpha+1)(\alpha+2)\cdots (\alpha+n-1) ,\quad (\alpha)_0=1
\]
and we assume that $c \notin \mathbb{Z}_{\leq 0}$.  It is a standard fact that the hypergeometric series converges when $|z|<1$ and when ${\rm Re}(c-a-b)>0$ is also converges when $z=1$; see for instance \cite[Chapter 1]{Bailey}.

We start from the elementary Taylor expansion
\[
(1-y)^{-\beta}=\sum_{k=0}^{\infty} \frac{(\beta)_{k}}{k!}y^k, \quad |y|<1.
\]

\begin{lem}\label{lem: HGFnonLegendre}
	For $x\geq1$ we have
	\[
	\int_{1}^{x} (1-y^{-2})^{-1/n}dy  = x\HGF[-\frac{1}{2}, \frac{1}{n}; \frac{1}{2}; x^{-2}] - \HGF[-\frac{1}{2}, \frac{1}{n}; \frac{1}{2}; 1].
	\]
\end{lem}
\begin{proof}
We compute
	\[
	\begin{aligned}
	\int_{1}^{x} (1-y^{-2})^{-1/n}dy  = \int_{1}^{x} \sum_{k=0}^{\infty} \frac{(\frac{1}{n})_{k}}{k!}y^{-2k}dy= y\sum_{k=0}^{\infty} \frac{(\frac{1}{n})_{k}}{k!}\frac{y^{-2k}}{(-2k+1)}\bigg|_{y=1}^{y=x}.
	\end{aligned}
	\]
The Lemma follows from the observation
$
	\frac{(-\frac{1}{2})_{k}}{(\frac{1}{2})_{k}} =- \frac{1}{2k-1}.
$
\end{proof}

\begin{prop}\label{lem: LegendreODE}
Let $n\geq 3$.	Suppose $g$ satisfies the ODE
	\begin{equation*}
	\frac{d}{dy} \left(\frac{g}{y}\right) = \frac{1}{y^2(1-y^n)^\frac{1}{2}}, \qquad 0 < y<1,
	\end{equation*}
	together with the initial conditions $g(0)=-1, g'(0)=0$.  Then we have
	\[
	g(y) = -\HGF[\frac{1}{2}, -\frac{1}{n}; \frac{n-1}{n}; y^n].
	\]
	In particular 
	\[
	g(1)= \lim_{y\to 1}g(y)= - \HGF[\frac{1}{2}, -\frac{1}{n}; \frac{n-1}{n}; 1]= -\frac{\Gamma(\frac{n-1}{n})\sqrt{\pi}}{\Gamma(\frac{n-2}{2n})}.
	\]
\end{prop}

\begin{proof}
Integrating
\[
y^{-2}(1-y^n)^{-1/2}=\sum_{k=0}^{\infty}\frac{(\frac{1}{2})_{k}}{k!}y^{nk-2},
\]
and utilizing the initial conditions to fix the leading coefficients,
\[
g(y)= \sum_{k=0}^{\infty}\frac{(\frac{1}{2})_{k}}{k!}\frac{y^{nk}}{nk-1}.
\]
The formula for $g(y)$ follows from the observation 
$
	\frac{-1}{nk-1} =  \frac{(-\frac{1}{n})_{k}}{(\frac{n-1}{n})_{k}}.
$
The evaluation of $g(1)$ appeals to 
Gauss' hypergeometric theorem \cite[Section 1.3]{Bailey}:  
\begin{lem}
	For $a,b,c$ with ${\rm Re}(c-a-b)>0$ we have
	\[
	\HGF[a,b;c;1] = \frac{\Gamma(c)\Gamma(c-a-b)}{\Gamma(c-a)\Gamma(c-b)}.
	\]
\end{lem}

\end{proof}

{\em Department of Mathematics, Massachusetts Institute of Technology, 77 Massachusetts Avenue, Cambridge, MA 02139}


\begin{thebibliography}{ }



\bibitem{Bailey} W. N. Bailey.  Generalized hypergeometric series. Cambridge Tracts in Mathematics and Mathematical Physics, No. 32, Stechert-Hafner, Inc. New York, 1964



\bibitem{BiquardDelcroix} 
Biquard, Olivier; Delcroix, Thibaut. Ricci flat Kähler metrics on rank two complex symmetric spaces. J. \'{E}c. polytech. Math. 6 (2019), 163--201.


\bibitem{Boucksom} 
Boucksom, Sébastien; Favre, Charles; Jonsson, Mattias. The non-Archimedean Monge-Ampère equation. Nonarchimedean and tropical geometry, 31--49, Simons Symp., Springer, [Cham], 2016.

\bibitem{CJY} Collins, Tristan C.; Jacob, Adam, Lin, Yu-Shen. Special Lagrangian submanifolds of log Calabi-Yau manifolds. Duke Math. J. 170 (2021), no. 7, 1291--1375

\bibitem{CJY2}  Collins, Tristan C.; Jacob, Adam, Lin, Yu-Shen. The SYZ mirror symmetry conjecture for del Pezzo surfaces and rational elliptic surfaces, preprint, arXiv:2012.05416

\bibitem{Ronan} 
Conlon, Ronan J.; Rochon, Frédéric. New examples of complete Calabi-Yau metrics on $\C^n$ for $n \geq 3$. Ann. Sci. \'{E}c. Norm. Sup\'er. (4) 54 (2021), no. 2, 259--303. 


\bibitem{GSVY} Greene, Brian R.; Shapere, Alfred D.; Vafa, Cumrun; Yau, Shing-Tung. Stringy cosmic strings and noncompact Calabi-Yau manifolds, Nuclear Phys. B, 337 (1990), no. 1, 1--36.

\bibitem{Heinthesis} 
Hein, Hans-Joachim. On gravitational instantons. Thesis (Ph.D.)–Princeton University. ProQuest LLC, Ann Arbor, MI, 2010. 129 pp.




\bibitem{HeinSobolev} 
Hein, Hans-Joachim. Weighted Sobolev inequalities under lower Ricci curvature bounds. Proc. Amer. Math. Soc. 139 (2011), no. 8, 2943--2955.


\bibitem{Heingravitational}
Hein, Hans-Joachim. Gravitational instantons from rational elliptic surfaces. J. Amer. Math. Soc. 25 (2012), no. 2, 355--393.




\bibitem{HSVZ} 
Hein, Hans-Joachim; Sun, Song; Viaclovsky, Jeff; Zhang, Ruobing. Nilpotent structures and collapsing Ricci-flat metrics on the K3 surface. J. Amer. Math. Soc. 35 (2021), no. 1, 123--209.


\bibitem{LiC3} 
Li, Yang. A new complete Calabi-Yau metric on $\C^3$. Invent. Math. 217 (2019), no. 1, 1--34.




\bibitem{LiTaubNUT}
Li, Yang.
SYZ geometry for Calabi-Yau 3-folds: Taub-NUT and Ooguri-Vafa type metrics. 	arXiv:1902.08770, accepted by AMS Memoir.


\bibitem{NASYZ} Li, Yang.
Metric SYZ conjecture and non-archimedean geometry. 	arXiv:2007.01384.



\bibitem{SYZ} 
Strominger, Andrew; Yau, Shing-Tung; Zaslow, Eric. Mirror symmetry is $T$-duality. Nuclear Phys. B 479 (1996), no. 1-2, 243--259. 




\bibitem{TianYau} 
Tian, G.; Yau, Shing-Tung. Complete Kähler manifolds with zero Ricci curvature. I. J. Amer. Math. Soc. 3 (1990), no. 3, 579--609. 

\bibitem{TianYau2}
Tian, G.; Yau, Shing-Tung. Complete Kähler manifolds with zero Ricci curvature. II., Invent. Math. 106 (1991), no.1, 27--60.



\bibitem{Tosatti} 
Tosatti, Valentino. Adiabatic limits of Ricci-flat Kähler metrics. J. Differential Geom. 84 (2010), no. 2, 427--453.

\bibitem{SunZhang} 
Song Sun, Ruobing Zhang. Complex structure degenerations and collapsing of Calabi-Yau metrics. 	arXiv:1906.03368.




\bibitem{Gabor}
Székelyhidi, Gábor. Degenerations of $\C^n$ and Calabi-Yau metrics. Duke Math. J. 168 (2019), no. 14, 2651--2700.


\bibitem{Gaboruniqueness} 
Székelyhidi, Gábor. Uniqueness of some Calabi-Yau metrics on ${\C}^n$. Geom. Funct. Anal. 30 (2020), no. 4, 1152--1182.

\bibitem{Yau}
Yau, Shing Tung. On the Ricci curvature of a compact Kähler manifold and the complex Monge-Ampère equation. I. Comm. Pure Appl. Math. 31 (1978), no. 3, 339--411.


\end{thebibliography}
\end{document}